\documentclass{amsart}
\usepackage{booktabs}
\usepackage{fullpage}
\usepackage{tristan}
\usepackage{algorithm}
\usepackage{algpseudocode}
\usepackage{todonotes}
\usepackage{placeins}

\usepackage{pgfplots}
\usepackage{pgfplotstable}
\usepackage{subcaption}
\usepackage{tikz}
\usepackage[
  giveninits=true,
  eprint=false,
  url=false,
  style=alphabetic,
  maxbibnames=99
]{biblatex}

\author[1,2]{Tristan Pryer}
  
\address{$^1$ Institute for Mathematical Innovation\\ University of
  Bath, Bath, UK. $^2$ Department of Mathematical Sciences
  \\ University of Bath, Bath, UK.}

\begin{document}

\newcommand{\spdeout}{./figures/}

\title{Polytopic symmetric interior-penalty approximation of elliptic
  SPDEs driven by spatial white noise }

\begin{abstract}
We analyse symmetric interior-penalty discontinuous Galerkin
approximations of elliptic stochastic partial differential equations
driven by spatial Gaussian white noise on quasi-uniform polytopic
meshes.  The low regularity of the exact random field precludes the
usual SIP consistency argument.  We instead restrict the same
continuum white-noise functional canonically to the discontinuous
polynomial space and separate the error into stochastic
forcing-restriction and deterministic discretisation components.

For every fixed polynomial degree $k\geq1$, we prove the sharp
two-sided estimate
\[
  \Norm{Y-Y_h}_{L^2\qp{\Xi;L^2(\Omega)}}
  \simeq
  h^{2-d/2},
\]
giving rates $h$ in two dimensions and $h^{1/2}$ in three dimensions.
A finite-dimensional lower bound establishes optimality.  The analysis
allows polytopic elements with arbitrarily many and arbitrarily small
faces, including admissible non-convex elements.  In two dimensions we
also treat polygons with a single reentrant corner.  Although the
deterministic approximation exhibits an angle-dependent loss of
regularity, the singular resolvent contribution is finite rank, and
the sharp first-order stochastic rate is retained.

The discontinuous structure yields an exact element-local sampler for
the restricted white noise.  Numerical experiments on irregular
polygonal, non-convex polytopic and reentrant domains support the
predicted convergence behaviour.
x\end{abstract}

\maketitle

\section{Introduction}
\label{sec:intro}

Spatial Gaussian white noise is a canonical model for irregular random
forcing and a fundamental ingredient in stochastic partial
differential equation representations of Gaussian random fields.  The
connection between Mat\'ern fields and elliptic SPDEs goes back to
Whittle and Mat\'ern and underpins the now standard SPDE approach to
spatial statistics
\cite{Whittle1954Plane,Matern1986,
LindgrenRueLindstrom2011SPDE,LindgrenBolinRue2022SPDEReview}.

We consider the representative problem
\begin{equation}
  \qp{-\Delta+\kappa^2}Y
  =
  \mathcal W
  \qquad
  \text{in }\Omega,
  \qquad
  Y|_{\partial\Omega}=0,
  \label{eq:intro_spde}
\end{equation}
where $\mathcal W$ denotes spatial Gaussian white noise.  The aim is
to approximate $Y$ by the classical symmetric interior-penalty
discontinuous Galerkin method on general polytopic meshes, while
retaining the original continuum white noise rather than replacing it
by an independently regularised forcing.

Numerical approximation of SPDEs with additive or white-noise forcing
has a substantial history.  Early finite-element, finite-difference
and lattice analyses include
\cite{AllenNovoselZhang1998,DuZhang2002,GyongyMartinez2006}.
For conforming finite element methods, strong and weak convergence,
fractional elliptic SPDEs and finite-dimensional representations of
the driving noise have subsequently been studied in
\cite{CaoYangYin2007,
ZhangRozovskiiKarniadakis2016,
BolinKirchnerKovacs2018,
CaoHongLiu2020,
BolinKirchnerKovacs2020FractionalSPDE}.
In particular, sharp strong rates of order $h^{2-d/2}$ are known for
conforming approximations of elliptic problems driven by spatial white
noise
\cite{ZhangRozovskiiKarniadakis2016,
CaoHongLiu2020,
BolinKirchnerKovacs2020FractionalSPDE}.
Related operator and Hilbert--Schmidt formulations for
Whittle--Mat\'ern fields are developed in
\cite{CoxKirchner2020}.

The representation and coupling of the Gaussian forcing form a
separate computational question.  Mass-matrix constructions for finite
element white noise and multilevel couplings, including couplings
between non-nested meshes, are developed in
\cite{CrociGilesRognesFarrell2018,
FairbanksVillaVassilevski2021WhiteNoise,
CrociGilesFarrell2021MLQMC}.

The principal analytical difficulty in \eqref{eq:intro_spde} is the
low spatial regularity of its solution.  White noise has regularity
below $H^{-d/2}(\Omega)$, and the elliptic resolvent gives, in the
regular case,
\[
  Y
  \in
  L^2\qp{\Xi;H^{2-d/2-\varepsilon}(\Omega)}
  \qquad
  \forall\varepsilon>0.
\]
Thus, in both dimensions considered here, the exact field need not
belong to $H^1_0(\Omega)$ and the usual energy-space formulation is
not the natural starting point.  Instead, the SPDE is defined through
the elliptic resolvent acting on an isonormal Gaussian process
\cite{BolinKirchnerKovacs2018,
BolinKirchnerKovacs2020FractionalSPDE,
LasanenRoininenHuttunen2018}.

This lack of regularity is particularly relevant for an interior
penalty method.  Classical SIP consistency is usually expressed by
inserting the exact solution into a broken bilinear form containing
elementwise gradients and normal traces.  Such quantities need not be
defined for $Y$.  Rather than assigning a DG differential structure
to the rough random field, we avoid this step completely.  We first
restrict $\mathcal W$ to the discontinuous polynomial space and solve
the corresponding continuum elliptic problem.  This produces the
decomposition
\[
  Y-Y_h
  =
  \qp{Y-Y^{(h)}}
  +
  \qp{Y^{(h)}-Y_h},
\]
in which the first term measures only restriction of the stochastic
forcing, while the second is an ordinary deterministic SIP error with
an $L^2(\Omega)$ right-hand side.  The stochastic and deterministic
parts can therefore be analysed using the tools appropriate to each.

DG discretisations of white-noise-driven problems have previously been
considered in more structured settings.  Existing analyses include LDG
approximation of elliptic SPDEs in two and three dimensions with
piecewise-constant regularisations of white noise on Cartesian meshes,
where the stochastic error is controlled through spatial increments of
the Green function \cite{YaoBo2007DGSPDE}, as well as finite element
and DG approximations of white-noise-driven Helmholtz and scattering
problems \cite{CaoZhangZhang2008Helmholtz,cao2008finite}.  Here we use
the classical SIP method on general polytopic meshes, couple every
discretisation to the same continuum white-noise functional, and seek
sharp strong convergence rates without introducing a separate
regularisation of the driving noise.

Our first main result establishes the sharp strong convergence rate
\begin{equation}
  \Norm{Y-Y_h}_{L^2\qp{\Xi;L^2(\Omega)}}
  \simeq
  h^{2-d/2}
  \label{eq:intro_main_result}
\end{equation}
for fixed polynomial degree on quasi-uniform polytopic meshes. A
matching finite-dimensional lower bound shows that these rates are
optimal for approximation spaces of comparable dimension.

Our second main result extends the analysis in two dimensions to
polytopic domains with a single reentrant corner.  Although the
deterministic SIP approximation reflects the associated loss of
elliptic regularity, we prove that the stochastic approximation
retains the sharp first-order rate
\begin{equation}
  \Norm{
    Y-Y_h
  }_{L^2\qp{\Xi;L^2(\Omega)}}
  \simeq
  h.
  \label{eq:intro_reentrant_result}
\end{equation}
The proof exploits the finite-dimensional structure of the corner
singularity and shows that it does not alter the leading stochastic
approximation error.

The analysis is carried out on general polytopic meshes.  Elements are
uniformly star-shaped with respect to a ball in their kernel, but may
have arbitrarily many and arbitrarily small faces and may be
non-convex.  The resulting estimates are uniform with respect to the
number and relative sizes of the faces \cite{botti2025trace,
  CangianiGeorgoulisHouston2014,
  CangianiDongGeorgoulisHouston2016ADRPolytopic,
  CangianiDongGeorgoulis2022Arbitrary}.

A further computational consequence of the discontinuous formulation
is that the canonical restricted white-noise load can be sampled
exactly from independent element-local Gaussian vectors.  The same
construction gives exact couplings across agglomerated mesh
hierarchies and is also used in the multilevel Monte-Carlo study
\cite{ChronholmEvansGeorgoulisPryer2026MLMC}.

The numerical experiments examine each of these features.  Direct
continuum-error calculations on regular two- and three-dimensional
domains test the rates in \eqref{eq:intro_main_result} across several
polytopic mesh families.  On reentrant domains, a corner-angle sweep
tests the robustness of the first-order stochastic rate in
\eqref{eq:intro_reentrant_result} with respect to the strength of the
corner singularity.

The remainder of the paper is organised as follows.
Section~\ref{sec:problem} introduces the white-noise-driven elliptic
problem and its spectral representation.
Section~\ref{sec:polytopic_dg} develops the polytopic SIP
discretisation, the canonical white-noise restriction and the sharp
strong convergence theory.
Section~\ref{sec:reentrant_polygon} treats the reentrant-corner case.
Section~\ref{sec:numerics} presents the numerical experiments.

\section{Problem setup}
\label{sec:problem}

Let $\Omega\subset\mathbb R^d$, $d\in\qc{2,3}$, be a bounded polytopic
Lipschitz domain, and let $(\Xi,\mathcal F,\mathbb P)$ be a complete
probability space.  We write $L^p(\Xi;X)$ for the usual Bochner spaces
and denote expectation by $\mathbb E$.

For $\kappa>0$, define
\begin{equation}
  a(u,v)
  :=
  \qp{
    \nabla u,\nabla v
  }_{L^2(\Omega)}
  +
  \kappa^2
  \qp{
    u,v
  }_{L^2(\Omega)},
  \qquad
  u,v\in H^1_0(\Omega),
  \label{eq:continuous_bilinear_form}
\end{equation}
and let $\Lop$ be the positive self-adjoint operator in $L^2(\Omega)$
associated with $a$.  Thus, $\Lop$ is the Dirichlet realisation of
$-\Delta+\kappa^2$.  We consider the white-noise-driven problem
\begin{equation}
  \Lop Y
  =
  \mathcal W.
  \label{eq:spde}
\end{equation}

The principal analysis is developed under the following deterministic
regularity assumption.

\begin{assumption}[Elliptic regularity]
  \label{ass:elliptic_regular}
  The Dirichlet realisation of $\Lop$ satisfies
  \begin{equation}
    D(\Lop)
    =
    H^2(\Omega)\cap H^1_0(\Omega)
    \label{eq:operator_domain}
  \end{equation}
  and
  \begin{equation}
    \Norm{z}_{H^2(\Omega)}
    \le
    C_{\mathrm{reg}}
    \Norm{\Lop z}_{L^2(\Omega)}
    \qquad
    \forall z\in D(\Lop).
    \label{eq:elliptic_H2_regular}
  \end{equation}
\end{assumption}

Assumption~\ref{ass:elliptic_regular} holds, in particular, on bounded
convex polytopic domain.  The analysis uses
\eqref{eq:elliptic_H2_regular} directly rather than convexity.  Note
in \S \ref{sec:reentrant_polygon} we replace this assumption by an
explicit corner-singularity decomposition on reentrant polygons.

Since $\Omega$ is bounded, $\Lop^{-1}:L^2(\Omega)\to L^2(\Omega)$ is
compact.  Consequently, $\Lop$ admits an $L^2$-orthonormal
eigensystem $\qc{\lambda_j,\phi_j}_{j\ge1}$ satisfying
\begin{equation}
  \Lop\phi_j
  =
  \lambda_j\phi_j,
  \qquad
  \qp{
    \phi_i,\phi_j
  }_{L^2(\Omega)}
  =
  \delta_{ij}.
  \label{eq:elliptic_eigensystem}
\end{equation}
By Weyl's law, there exist constants $c_\lambda,C_\lambda>0$ such
that
\begin{equation}
  c_\lambda j^{2/d}
  \le
  \lambda_j
  \le
  C_\lambda j^{2/d},
  \qquad
  j\ge1.
  \label{eq:weyl}
\end{equation}
For $s\ge0$, let $\mathcal H^s_{\Lop}$ be the Hilbert space
\begin{equation}
  \mathcal H^s_{\Lop}
  :=
  \ensemble{
    v\in L^2(\Omega)
  }{
    \sum_{j=1}^\infty
    \lambda_j^s
    \norm{
      \qp{v,\phi_j}_{L^2(\Omega)}
    }^2
    <
    \infty
  },
  \label{eq:elliptic_hilbert_scale}
\end{equation}
equipped with the norm
\begin{equation}
  \Norm{v}_{\mathcal H^s_{\Lop}}^2
  :=
  \sum_{j=1}^\infty
  \lambda_j^s
  \norm{
    \qp{v,\phi_j}_{L^2(\Omega)}
  }^2.
  \label{eq:elliptic_hilbert_scale_norm}
\end{equation}
In particular, $\mathcal H^0_{\Lop}=L^2(\Omega)$ and $\mathcal
H^2_{\Lop}=D(\Lop)$ with equivalent norms.  Under
Assumption~\ref{ass:elliptic_regular}, interpolation yields
\begin{equation}
  \mathcal H^s_{\Lop}
  \hookrightarrow
  H^s(\Omega),
  \qquad
  0\le s\le2.
  \label{eq:elliptic_scale_sobolev_embedding}
\end{equation}

\subsection{Spatial white noise and the elliptic solution}
\label{subsec:white_noise_problem}

We regard $\mathcal W$ as an isonormal Gaussian process over
$L^2(\Omega)$. It is a centred Gaussian linear mapping $\mathcal
W:L^2(\Omega)\to L^2(\Xi)$ satisfying
\begin{equation}
  \mathbb E\qb{
    \mathcal W(v)\mathcal W(w)
  }
  =
  \qp{
    v,w
  }_{L^2(\Omega)}
  \qquad
  \forall v,w\in L^2(\Omega).
  \label{eq:white_noise_covariance}
\end{equation}
In particular,
\begin{equation}
  \xi_j
  :=
  \mathcal W(\phi_j),
  \qquad
  j\ge1,
  \label{eq:white_noise_coefficients}
\end{equation}
are independent standard Gaussian random variables.  The solution of
\eqref{eq:spde} is defined spectrally by
\begin{equation}
  Y
  :=
  \sum_{j=1}^\infty
  \lambda_j^{-1}\xi_j\phi_j.
  \label{eq:spde_spectral_solution}
\end{equation}

\begin{theorem}[Regularity of the white-noise solution]
  \label{thm:spde_regular}
  Let $d\in\qc{2,3}$ and suppose that the eigenvalues of $\Lop$
  satisfy \eqref{eq:weyl}.  Then the series
  \eqref{eq:spde_spectral_solution} converges in
  $L^2\qp{\Xi;L^2(\Omega)}$ and defines the unique
  $L^2(\Omega)$-valued solution of \eqref{eq:spde} in the resolvent
  sense.  Moreover,
  \begin{equation}
    Y
    \in
    L^2\qp{
      \Xi;
      \mathcal H^s_{\Lop}
    }
    \qquad
    \forall s<2-\frac d2,
    \label{eq:Y_elliptic_regular}
  \end{equation}
  whereas
  \begin{equation}
    \mathbb E\qb{
      \Norm{Y}_{\mathcal H^{2-d/2}_{\Lop}}^2
    }
    =
    \infty.
    \label{eq:Y_critical_divergence}
  \end{equation}

  If Assumption~\ref{ass:elliptic_regular} also holds, then
  \begin{equation}
    Y
    \in
    L^2\qp{
      \Xi;
      H^s(\Omega)
    }
    \qquad
    \forall\,0\le s<2-\frac d2.
    \label{eq:Y_regular}
  \end{equation}
  Thus the available Sobolev regularity is $H^{1-\varepsilon}$ in two
  dimensions and $H^{1/2-\varepsilon}$ in three dimensions.
\end{theorem}

\begin{proof}
  This is the case $\beta=1$ and eigenvalue-growth exponent
  $\alpha=2/d$ of
  \cite[Proposition~2.3 and Remark~2.4]
       {BolinKirchnerKovacs2020FractionalSPDE}.
  Indeed, for every $s\in\mathbb R$,
  \begin{equation}
    \mathbb E\qb{
      \Norm{Y}_{\mathcal H^s_{\Lop}}^2
    }
    =
    \sum_{j=1}^\infty
    \lambda_j^{s-2}.
    \label{eq:Y_Hs_moment}
  \end{equation}
  By \eqref{eq:weyl}, this series has the same convergence behaviour
  as $\sum_{j=1}^\infty j^{2(s-2)/d},$ and is therefore finite
  precisely when $s<2-d/2$.  The choice $s=0$ gives convergence in
  $L^2\qp{\Xi;L^2(\Omega)}$, while $s=2-d/2$ gives the critical
  logarithmic divergence.  Finally, \eqref{eq:Y_regular} follows from
  \eqref{eq:elliptic_scale_sobolev_embedding}.
\end{proof}

\section{Polytopic symmetric interior-penalty discretisation}
\label{sec:polytopic_dg}

The convergence analysis separates restriction of the stochastic
forcing from approximation of the elliptic resolvent.  Let $\Wh$ be
the canonical restriction of spatial white noise to a finite element
space $V_h$, defined below, and let $Y^{(h)}$ solve $\Lop
Y^{(h)}=\Wh$.  Then
\begin{equation}
  Y-Y_h
  =
  \qp{Y-Y^{(h)}}
  +
  \qp{Y^{(h)}-Y_h}.
  \label{eq:basic_error_split}
\end{equation}
The first term measures restriction of the white noise, while the
second is a deterministic SIP error with an $L^2$ right-hand side.

\subsection{Polytopic meshes and discrete spaces}
\label{subsec:poly_mesh}

Let $\mathcal T_h$ be a partition of $\Omega$ into non-overlapping open
polytopes $K$ satisfying
\begin{equation}
  \overline\Omega
  =
  \bigcup_{K\in\mathcal T_h}\overline K,
  \qquad
  K\cap K'
  =
  \emptyset
  \quad
  \text{for }K\neq K'.
  \label{eq:mesh_partition}
\end{equation}
For $K\in\mathcal T_h$, let
$h_K:=\operatorname{diam}(K)$ and set
\begin{equation}
  h
  :=
  \max_{K\in\mathcal T_h}h_K.
  \label{eq:mesh_size}
\end{equation}

Let $\mathcal F_h=\mathcal F_h^i\cup\mathcal F_h^b$ be the set of open
$(d-1)$-dimensional mesh faces.  For
$F=\partial K^-\cap\partial K^+\in\mathcal F_h^i$, fix a unit normal
$\vec n_F$ pointing from $K^-$ to $K^+$; on boundary faces,
$\vec n_F$ is the outward unit normal.  We use the face scale
\begin{equation}
  h_F
  :=
  \begin{cases}
    \displaystyle
    \frac{2h_{K^-}h_{K^+}}
         {h_{K^-}+h_{K^+}},
    &
    F=\partial K^-\cap\partial K^+
    \in\mathcal F_h^i,
    \\[3mm]
    h_K,
    &
    F\subset\partial K
    \in\mathcal F_h^b.
  \end{cases}
  \label{eq:face_scale}
\end{equation}

For a piecewise scalar field $v$, define on interior faces
\begin{equation}
  \jump{v}
  :=
  v^--v^+,
  \qquad
  \avg{v}
  :=
  \frac12\qp{v^-+v^+},
  \label{eq:jump_average}
\end{equation}
and set $\jump{v}:=v$ and $\avg{v}:=v$ on boundary faces.  Averages
of vector fields are taken componentwise.

\begin{assumption}[Polytopic mesh regularity]
  \label{ass:poly_mesh}
  There exists $\vartheta_\ast\in(0,1]$, independent of $h$, such
  that every $K\in\mathcal T_h$ contains a ball
  $B(x_K,\rho_K)$ in its kernel satisfying
  \begin{equation}
    \rho_K
    \ge
    \vartheta_\ast h_K.
    \label{eq:poly_shape_reg}
  \end{equation}
  Equivalently,
  \begin{equation}
    B(x_K,\rho_K)\subset K,
    \qquad
    \qb{x,y}\subset K
    \quad
    \forall x\in B(x_K,\rho_K),
    \quad
    \forall y\in K.
    \label{eq:poly_kernel_ball}
  \end{equation}
  Moreover, the mesh family is quasi-uniform in that there exists
  $c_{\mathrm{qu}}>0$, independent of $h$, such that
  \begin{equation}
    c_{\mathrm{qu}}h
    \le
    h_K
    \le
    h
    \qquad
    \forall K\in\mathcal T_h.
    \label{eq:quasi_uniform}
  \end{equation}
\end{assumption}

Assumption~\ref{ass:poly_mesh} implies
\begin{equation}
  \abs{K}
  \simeq
  h_K^d,
  \label{eq:element_volume_regular}
\end{equation}
with constants depending only on $\vartheta_\ast$ and $d$.
It also implies the geometric condition required by
\cite[Lemma~2.2]{botti2025trace}.  Indeed, triangulating each face and
joining the resulting $(d-1)$-simplices $F_T$ to $x_K$ gives
\begin{equation}
  \frac{d\abs{T}}{\abs{F_T}}
  =
  \operatorname{dist}
  \qp{x_K,\operatorname{aff}(F_T)}
  \ge
  \rho_K
  \ge
  \vartheta_\ast h_K.
  \label{eq:botti_geometric_condition}
\end{equation}
Hence condition~(7) of \cite{botti2025trace} holds uniformly.  In
particular, no uniform lower bound is required on the size of an
individual face, and the number of faces per element may be
arbitrarily large.

The quasi-uniformity condition \eqref{eq:quasi_uniform} is actually
not needed for the local estimates below; it will be used subsequently
to pass from local mesh scales to $h$ and to obtain $N_h\simeq
h^{-d}$.

Introduce the broken Sobolev space
\begin{equation}
  H^1(\mathcal T_h)
  :=
  \ensemble{
    v\in L^2(\Omega)
  }{
    v|_K\in H^1(K)
    \quad
    \forall K\in\mathcal T_h
  },
  \label{eq:broken_h1}
\end{equation}
with broken gradient
\begin{equation}
  \qp{\nabla_hv}|_K
  :=
  \nabla\qp{v|_K}.
  \label{eq:broken_gradient}
\end{equation}

\begin{lemma}[Trace, inverse and approximation estimates]
  \label{lem:poly_trace_inverse}
  Under the first part of
  Assumption~\ref{ass:poly_mesh},
  \begin{equation}
    \Norm{v}_{L^2(\partial K)}^2
    \lesssim
    h_K^{-1}
    \Norm{v}_{L^2(K)}^2
    +
    h_K
    \Norm{\nabla v}_{L^2(K)}^2
    \qquad
    \forall v\in H^1(K).
    \label{eq:poly_trace}
  \end{equation}

  For every fixed polynomial degree $k\ge1$,
  \begin{align}
    \Norm{\nabla v_h}_{L^2(K)}
    &\lesssim
    h_K^{-1}
    \Norm{v_h}_{L^2(K)},
    \label{eq:poly_inverse}
    \\
    \Norm{v_h}_{L^2(\partial K)}
    &\lesssim
    h_K^{-1/2}
    \Norm{v_h}_{L^2(K)}
    \label{eq:poly_trace_inverse}
  \end{align}
  for all $v_h\in\mathbb P_k(K)$.  Consequently,
  \begin{equation}
    \sum_{F\subset\partial K}
    h_K
    \Norm{
      \vec\tau_h\cdot\vec n_F
    }_{L^2(F)}^2
    \lesssim
    \Norm{
      \vec\tau_h
    }_{L^2(K)}^2
    \qquad
    \forall
    \vec\tau_h\in\qb{\mathbb P_k(K)}^d.
    \label{eq:poly_discrete_trace}
  \end{equation}

  Let $\Pih$ be the elementwise $L^2$-orthogonal projection, defined
  by
  \[
    \qp{\Pih v}|_K
    =
    \Pi_K\qp{v|_K},
    \qquad
    \Pi_K
    :
    L^2(K)
    \longrightarrow
    \mathbb P_k(K).
  \]
  Then, for every $v\in H^2(K)$,
  \begin{equation}
    \begin{split}
      &
      \Norm{v-\Pi_Kv}_{L^2(K)}
      +
      h_K
      \Norm{
        \nabla\qp{v-\Pi_Kv}
      }_{L^2(K)}
      +
      h_K^2
      \Norm{
        D^2\qp{v-\Pi_Kv}
      }_{L^2(K)}
      \\
      &\qquad
      +
      h_K^{1/2}
      \Norm{
        v-\Pi_Kv
      }_{L^2(\partial K)}
      +
      h_K^{3/2}
      \Norm{
        \nabla\qp{v-\Pi_Kv}
      }_{L^2(\partial K)}
      \lesssim
      h_K^2
      \Norm{v}_{H^2(K)}.
    \end{split}
    \label{eq:poly_projection_approximation}
  \end{equation}
  Consequently, 
  \begin{equation}
    \Norm{
      v-\Pih v
    }_{L^2(\Omega)}
    \lesssim
    h^2
    \Norm{v}_{H^2(\Omega)}
    \qquad
    \forall v\in H^2(\Omega).
    \label{eq:L2_projection_H2}
  \end{equation}

  All hidden constants are independent of the number and relative
  sizes of the faces.
\end{lemma}

\begin{proof}
  By \eqref{eq:botti_geometric_condition}, the mesh condition of
  \cite[Section~2.2]{botti2025trace} holds uniformly.  Taking
  $p=q=2$ in \cite[Lemma~2.2]{botti2025trace} gives
  \[
    \Norm{v}_{L^2(\partial K)}^2
    \lesssim
    h_K^{-1}
    \Norm{v}_{L^2(K)}^2
    +
    \Norm{\nabla v}_{L^2(K)}
    \Norm{v}_{L^2(K)}.
  \]
  Young's inequality with the scale $h_K$ yields
  \eqref{eq:poly_trace}.

  The polynomial inverse and approximation estimates on uniformly
  star-shaped polytopes are standard; see, for example,
  \cite{CangianiGeorgoulisHouston2014,
    CangianiDongGeorgoulisHouston2016ADRPolytopic,
    CangianiDongGeorgoulis2022Arbitrary}.
  Combining \eqref{eq:poly_trace} with
  \eqref{eq:poly_inverse} gives
  \eqref{eq:poly_trace_inverse}.  Moreover,
  \[
    \sum_{F\subset\partial K}
    \Norm{
      \vec\tau_h\cdot\vec n_F
    }_{L^2(F)}^2
    \leq
    \Norm{
      \vec\tau_h
    }_{L^2(\partial K)}^2,
  \]
  so \eqref{eq:poly_discrete_trace} follows immediately.

  For \eqref{eq:poly_projection_approximation}, choose a polynomial
  $p_K\in\mathbb P_k(K)$ satisfying the standard $H^2$
  approximation estimates.  The $L^2$ best-approximation property
  of $\Pi_K$, followed by the polynomial inverse estimate applied to
  $p_K-\Pi_Kv$, gives the three volume estimates.  Applying
  \eqref{eq:poly_trace} first to $v-\Pi_Kv$ and then to each
  component of $\nabla(v-\Pi_Kv)$ gives the two boundary estimates.
  Summing the local $L^2$ estimate and using
  \eqref{eq:quasi_uniform} proves
  \eqref{eq:L2_projection_H2}.
\end{proof}

For fixed $k\ge1$, define
\begin{equation}
  V_h
  :=
  \ensemble{
    v_h\in L^2(\Omega)
  }{
    v_h|_K\in\mathbb P_k(K)
    \quad
    \forall K\in\mathcal T_h
  }.
  \label{eq:dg_space}
\end{equation}
By \eqref{eq:element_volume_regular},
\eqref{eq:quasi_uniform} and the fixed polynomial degree,
\begin{equation}
  N_h
  :=
  \dim V_h
  \simeq
  h^{-d}.
  \label{eq:dimension_bound}
\end{equation}

\subsection{SIP form and discrete operator}
\label{subsec:dg_form}

Introduce the broken vector-polynomial space
\begin{equation}
  \Sigma_h
  :=
  \ensemble{
    \vec\tau_h\in L^2\qp{\Omega;\mathbb R^d}
  }{
    \vec\tau_h|_K\in\qb{\mathbb P_k(K)}^d
    \quad
    \forall K\in\mathcal T_h
  }.
  \label{eq:lifting_space}
\end{equation}
The lifting $\Rh:L^2(\mathcal F_h)\to\Sigma_h$ is defined by
\begin{equation}
  \qp{
    \Rh(\varphi),\vec\tau_h
  }_{L^2(\Omega)}
  =
  -
  \sum_{F\in\mathcal F_h}
  \qp{
    \varphi,
    \avg{
      \vec\tau_h\cdot\vec n_F
    }
  }_{L^2(F)}
  \qquad
  \forall\vec\tau_h\in\Sigma_h.
  \label{eq:lifting}
\end{equation}
By \eqref{eq:poly_discrete_trace},
\begin{equation}
  \Norm{
    \Rh(\varphi)
  }_{L^2(\Omega)}
  \lesssim
  \qp{
    \sum_{F\in\mathcal F_h}
    h_F^{-1}
    \Norm{\varphi}_{L^2(F)}^2
  }^{1/2}.
  \label{eq:lifting_stability}
\end{equation}

For $\gamma>0$, set
\begin{equation}
  \pen{F}
  :=
  \frac{\gamma k^2}{h_F}.
  \label{eq:penalty_parameter}
\end{equation}
The SIP form $a_h:V_h\times V_h\to\mathbb R$ is
\begin{equation}
  \begin{split}
    a_h(w_h,v_h)
    &:=
    \qp{
      \nabla_hw_h,\nabla_hv_h
    }_{L^2(\Omega)}
    +
    \kappa^2
    \qp{
      w_h,v_h
    }_{L^2(\Omega)}
    \\
    &\quad
    -
    \sum_{F\in\mathcal F_h}
    \qp{
      \avg{
        \nabla_hw_h\cdot\vec n_F
      },
      \jump{v_h}
    }_{L^2(F)}
    \\
    &\quad
    -
    \sum_{F\in\mathcal F_h}
    \qp{
      \avg{
        \nabla_hv_h\cdot\vec n_F
      },
      \jump{w_h}
    }_{L^2(F)}
    \\
    &\quad
    +
    \sum_{F\in\mathcal F_h}
    \pen{F}
    \qp{
      \jump{w_h},\jump{v_h}
    }_{L^2(F)}.
  \end{split}
  \label{eq:sip_form}
\end{equation}
Since $\nabla_hV_h\subset\Sigma_h$, the two consistency terms in
\eqref{eq:sip_form} may equivalently be represented through the
lifting.  For example,
\[
  -
  \sum_{F\in\mathcal F_h}
  \qp{
    \avg{
      \nabla_hw_h\cdot\vec n_F
    },
    \jump{v_h}
  }_{L^2(F)}
  =
  \qp{
    \Rh\qp{\jump{v_h}},
    \nabla_hw_h
  }_{L^2(\Omega)}.
\]
Thus the lifting estimate \eqref{eq:lifting_stability} controls the
face terms without any dependence on the number or relative sizes of
the faces.

For $v_h\in V_h$, define
\begin{equation}
  \dgnorm{h}{v_h}^2
  :=
  \Norm{\nabla_hv_h}_{L^2(\Omega)}^2
  +
  \kappa^2
  \Norm{v_h}_{L^2(\Omega)}^2
  +
  \sum_{F\in\mathcal F_h}
  \pen{F}
  \Norm{\jump{v_h}}_{L^2(F)}^2.
  \label{eq:dg_norm}
\end{equation}

\begin{lemma}[SIP stability]
  \label{lem:sip_stability}
  For $\gamma$ sufficiently large,
  \begin{equation}
    a_h(v_h,v_h)
    \gtrsim
    \dgnorm{h}{v_h}^2,
    \qquad
    \norm{
      a_h(w_h,v_h)
    }
    \lesssim
    \dgnorm{h}{w_h}
    \dgnorm{h}{v_h}
    \label{eq:sip_stability}
  \end{equation}
  for all $w_h,v_h\in V_h$, with constants independent of $h$ and
  of the number and relative sizes of the element faces.
\end{lemma}

\begin{proof}
  Represent the consistency terms through $\Rh$ and apply
  \eqref{eq:lifting_stability}.  Cauchy--Schwarz and Young's
  inequality then give coercivity for $\gamma$ sufficiently large
  and continuity in the DG norm.
\end{proof}

Define $L_h:V_h\to V_h$ by
\begin{equation}
  \qp{
    L_hw_h,v_h
  }_{L^2(\Omega)}
  =
  a_h(w_h,v_h)
  \qquad
  \forall w_h,v_h\in V_h.
  \label{eq:discrete_operator}
\end{equation}
Lemma~\ref{lem:sip_stability} shows that $L_h$ is symmetric positive
definite.

\subsection{Deterministic resolvent approximation}
\label{subsec:deterministic_L2}

For the consistency analysis under
Assumption~\ref{ass:elliptic_regular}, set
\begin{equation}
  \mathcal V_h^\sharp
  :=
  \qp{
    H^2(\Omega)\cap H^1_0(\Omega)
  }
  +
  V_h.
  \label{eq:extended_sip_space}
\end{equation}
For $w,v\in\mathcal V_h^\sharp$, let $a_h^\sharp(w,v)$ be given by
the right-hand side of \eqref{eq:sip_form}, with $(w_h,v_h)$ replaced
by $(w,v)$.  The face terms are well-defined by the $H^2$ trace
regularity, and the restriction of $a_h^\sharp$ to
$V_h\times V_h$ is precisely $a_h$.

Set
\begin{equation}
  \begin{split}
    \Norm{v}_{h,\ast}^2
    &:=
    \Norm{\nabla_hv}_{L^2(\Omega)}^2
    +
    \kappa^2
    \Norm{v}_{L^2(\Omega)}^2
    +
    \sum_{F\in\mathcal F_h}
    \pen{F}
    \Norm{\jump{v}}_{L^2(F)}^2
    \\
    &\quad
    +
    \sum_{F\in\mathcal F_h}
    h_F
    \Norm{
      \avg{
        \nabla_hv\cdot\vec n_F
      }
    }_{L^2(F)}^2.
  \end{split}
  \label{eq:extended_dg_norm}
\end{equation}
Cauchy--Schwarz gives
\begin{equation}
  \norm{
    a_h^\sharp(w,v)
  }
  \lesssim
  \Norm{w}_{h,\ast}
  \Norm{v}_{h,\ast}
  \qquad
  \forall w,v\in\mathcal V_h^\sharp.
  \label{eq:extended_sip_continuity}
\end{equation}
Moreover, Lemma~\ref{lem:poly_trace_inverse} implies
\begin{align}
  \Norm{v_h}_{h,\ast}
  &\lesssim
  \dgnorm{h}{v_h}
  \qquad
  \forall v_h\in V_h,
  \label{eq:discrete_extended_norm_control}
  \\
  \Norm{v-\Pih v}_{h,\ast}
  &\lesssim
  h\Norm{v}_{H^2(\Omega)}
  \qquad
  \forall v\in H^2(\Omega)\cap H^1_0(\Omega).
  \label{eq:extended_projection_approximation}
\end{align}

\begin{lemma}[Deterministic resolvent approximation]
  \label{lem:deterministic_L2}
  Suppose Assumptions~\ref{ass:elliptic_regular} and
  \ref{ass:poly_mesh} hold, and let $\gamma$ be sufficiently large for
  Lemma~\ref{lem:sip_stability}.  For $f\in L^2(\Omega)$, let
  $z=\Lop^{-1}f$ and let $z_h\in V_h$ satisfy
  \begin{equation}
    a_h(z_h,v_h)
    =
    \qp{
      f,v_h
    }_{L^2(\Omega)}
    \qquad
    \forall v_h\in V_h.
    \label{eq:deterministic_dg_problem}
  \end{equation}
  Then
  \begin{equation}
    \Norm{
      z-z_h
    }_{L^2(\Omega)}
    \lesssim
    h^2
    \Norm{
      f
    }_{L^2(\Omega)}.
    \label{eq:deterministic_L2_error}
  \end{equation}
  Equivalently,
  \begin{equation}
    \Norm{
      \Lop^{-1}
      -
      L_h^{-1}\Pih
    }_{\mathcal L\qp{L^2(\Omega),L^2(\Omega)}}
    \lesssim
    h^2.
    \label{eq:resolvent_error}
  \end{equation}
\end{lemma}

\begin{proof}
  Elliptic regularity gives
  \[
    \Norm{z}_{H^2(\Omega)}
    \lesssim
    \Norm{f}_{L^2(\Omega)}.
  \]
  Since $z\in H^2(\Omega)\cap H^1_0(\Omega)$, elementwise integration
  by parts yields
  \begin{equation}
    a_h^\sharp(z,v_h)
    =
    \qp{
      f,v_h
    }_{L^2(\Omega)}
    \qquad
    \forall v_h\in V_h.
    \label{eq:exact_sip_consistency_identity}
  \end{equation}
  Hence
  \[
    a_h^\sharp(z-z_h,v_h)
    =
    0
    \qquad
    \forall v_h\in V_h.
  \]
  Let
  \[
    \eta
    :=
    z-\Pih z,
    \qquad
    \xi_h
    :=
    \Pih z-z_h.
  \]
  Galerkin orthogonality gives
  \[
    a_h(\xi_h,v_h)
    =
    -
    a_h^\sharp(\eta,v_h)
    \qquad
    \forall v_h\in V_h.
  \]
  Taking $v_h=\xi_h$ and using Lemma~\ref{lem:sip_stability},
  \eqref{eq:extended_sip_continuity},
  \eqref{eq:discrete_extended_norm_control} and
  \eqref{eq:extended_projection_approximation}, we obtain
  \[
    \dgnorm{h}{\xi_h}
    \lesssim
    h\Norm{z}_{H^2(\Omega)}.
  \]
  Therefore, with $e:=z-z_h=\eta+\xi_h$,
  \begin{equation}
    \Norm{e}_{h,\ast}
    \lesssim
    h\Norm{z}_{H^2(\Omega)}.
    \label{eq:deterministic_energy_error}
  \end{equation}

  For the $L^2$ estimate, let
  \[
    \psi
    :=
    \Lop^{-1}e.
  \]
  By Assumption~\ref{ass:elliptic_regular},
  \[
    \Norm{\psi}_{H^2(\Omega)}
    \lesssim
    \Norm{e}_{L^2(\Omega)}.
  \]
  Symmetry, consistency and Galerkin orthogonality give
  \[
    \Norm{e}_{L^2(\Omega)}^2
    =
    a_h^\sharp(e,\psi)
    =
    a_h^\sharp
    \qp{
      e,\psi-\Pih\psi
    }.
  \]
  Hence, by
  \eqref{eq:extended_sip_continuity},
  \eqref{eq:extended_projection_approximation} and
  \eqref{eq:deterministic_energy_error},
  \begin{align*}
    \Norm{e}_{L^2(\Omega)}^2
    &\lesssim
    \Norm{e}_{h,\ast}
    \Norm{\psi-\Pih\psi}_{h,\ast}
    \\
    &\lesssim
    h^2
    \Norm{z}_{H^2(\Omega)}
    \Norm{\psi}_{H^2(\Omega)}
    \\
    &\lesssim
    h^2
    \Norm{f}_{L^2(\Omega)}
    \Norm{e}_{L^2(\Omega)}.
  \end{align*}
  This proves \eqref{eq:deterministic_L2_error}.

  Finally,
  \[
    \qp{
      f,v_h
    }_{L^2(\Omega)}
    =
    \qp{
      \Pih f,v_h
    }_{L^2(\Omega)}
    \qquad
    \forall v_h\in V_h,
  \]
  so $z_h=L_h^{-1}\Pih f$, and
  \eqref{eq:resolvent_error} follows.
\end{proof}

\subsection{Canonical discrete white noise}
\label{subsec:discrete_white_noise}

The restriction of $\mathcal W$ to $V_h$ is represented by the unique
Gaussian random variable $\Wh\in L^2(\Xi;V_h)$ satisfying
\begin{equation}
  \qp{
    \Wh,v_h
  }_{L^2(\Omega)}
  =
  \mathcal W(v_h)
  \qquad
  \forall v_h\in V_h.
  \label{eq:discrete_white_noise}
\end{equation}
Its covariance is
\begin{equation}
  \mathbb E\qb{
    \qp{\Wh,v_h}_{L^2(\Omega)}
    \qp{\Wh,w_h}_{L^2(\Omega)}
  }
  =
  \qp{
    v_h,w_h
  }_{L^2(\Omega)}
  \qquad
  \forall v_h,w_h\in V_h.
  \label{eq:discrete_white_noise_covariance}
\end{equation}

Let $\qc{\phi_{K,i}}_{i=1}^{N_k}$ be a basis of
$\mathbb P_k(K)$ and let
\begin{equation*}
  \qp{M_K}_{ij}
  :=
  \int_K
  \phi_{K,i}\phi_{K,j}
  \,\mathrm d x
\end{equation*}
be the local mass matrix.  The corresponding load vector
\begin{equation*}
  \vec b_K
  :=
  \qp{
    \mathcal W(\phi_{K,1}),
    \ldots,
    \mathcal W(\phi_{K,N_k})
  }^\top
\end{equation*}
is Gaussian with covariance $M_K$.  It can therefore be sampled as
\begin{equation}
  \vec b_K
  =
  M_K^{1/2}\vec\xi_K,
  \qquad
  \vec\xi_K
  \sim
  \mathcal N\qp{0,I}.
  \label{eq:local_white_noise_sampling}
\end{equation}
If $\vec c_K$ denotes the coefficient vector of $\Wh|_K$, then
$M_K\vec c_K=\vec b_K$.  Basis functions on distinct elements have
disjoint support, so the vectors $\vec b_K$ are mutually independent.
Thus \eqref{eq:local_white_noise_sampling} gives an exact element-local
sampler for the continuum white-noise load restricted to $V_h$.

If $\qc{\psi_i}_{i=1}^{N_h}$ is an $L^2$-orthonormal basis of $V_h$,
then
\[
  \Wh
  =
  \sum_{i=1}^{N_h}
  \zeta_i\psi_i,
  \qquad
  \zeta_i
  \sim
  \mathcal N\qp{0,1},
\]
with independent coefficients.  Consequently,
\begin{equation}
  \mathbb E\qb{
    \Norm{\Wh}_{L^2(\Omega)}^2
  }
  =
  N_h
  \simeq
  h^{-d}.
  \label{eq:white_noise_dimension}
\end{equation}

The discrete SPDE is: find $Y_h\in V_h$ such that
\begin{equation}
  a_h(Y_h,v_h)
  =
  \mathcal W(v_h)
  =
  \qp{
    \Wh,v_h
  }_{L^2(\Omega)}
  \qquad
  \forall v_h\in V_h,
  \label{eq:discrete_spde}
\end{equation}
or, equivalently,
\begin{equation}
  L_hY_h
  =
  \Wh.
  \label{eq:discrete_spde_operator}
\end{equation}

\subsection{Restricted noise and strong convergence}
\label{subsec:mean_square_convergence}

Let $Y^{(h)}$ be the continuum solution driven by $\Wh$:
\begin{equation}
  \Lop Y^{(h)}
  =
  \Wh.
  \label{eq:projected_noise_continuum_problem}
\end{equation}
Since $\Wh\in L^2(\Xi;L^2(\Omega))$,
Assumption~\ref{ass:elliptic_regular} gives
\begin{equation}
  Y^{(h)}
  \in
  L^2\qp{
    \Xi;
    H^2(\Omega)\cap H^1_0(\Omega)
  }.
  \label{eq:projected_noise_regularity}
\end{equation}

\begin{lemma}[White-noise restriction error]
  \label{lem:projected_white_noise}
  For $d\in\qc{2,3}$,
  \begin{equation}
    \mathbb E\qb{
      \Norm{
        Y-Y^{(h)}
      }_{L^2(\Omega)}^2
    }
    =
    \sum_{j=1}^\infty
    \lambda_j^{-2}
    \Norm{
      \qp{I-\Pih}\phi_j
    }_{L^2(\Omega)}^2.
    \label{eq:projected_noise_spectral_error}
  \end{equation}
  Consequently,
  \begin{equation}
    \Norm{
      Y-Y^{(h)}
    }_{L^2\qp{\Xi;L^2(\Omega)}}
    \lesssim
    h^{2-d/2}.
    \label{eq:projected_noise_rate}
  \end{equation}
\end{lemma}

\begin{proof}
  For every eigenfunction $\phi_j$,
  \[
    \qp{
      \Wh,\phi_j
    }_{L^2(\Omega)}
    =
    \qp{
      \Wh,\Pih\phi_j
    }_{L^2(\Omega)}
    =
    \mathcal W(\Pih\phi_j).
  \]
  Therefore
  \[
    \qp{
      Y-Y^{(h)},\phi_j
    }_{L^2(\Omega)}
    =
    \lambda_j^{-1}
    \mathcal W\qp{
      \qp{I-\Pih}\phi_j
    }.
  \]
  Parseval's identity and the covariance of $\mathcal W$ give
  \eqref{eq:projected_noise_spectral_error}.

  By \eqref{eq:L2_projection_H2},
  Assumption~\ref{ass:elliptic_regular} and the $L^2$ stability of
  $\Pih$,
  \[
    \Norm{
      \qp{I-\Pih}\phi_j
    }_{L^2(\Omega)}
    \lesssim
    \min\qc{
      1,
      h^2\lambda_j
    }.
  \]
  Hence
  \begin{equation}
    \mathbb E\qb{
      \Norm{
        Y-Y^{(h)}
      }_{L^2(\Omega)}^2
    }
    \lesssim
    \sum_{j=1}^\infty
    \min\qc{
      \lambda_j^{-2},
      h^4
    }.
    \label{eq:spectral_sum_min}
  \end{equation}
  Using \eqref{eq:weyl} and splitting the sum at
  $j\simeq h^{-d}$ gives
  \begin{equation}
    \sum_{j\lesssim h^{-d}}
    h^4
    +
    \sum_{j\gtrsim h^{-d}}
    \lambda_j^{-2}
    \lesssim
    h^{4-d}.
    \label{eq:spectral_sum_bound}
  \end{equation}
  Taking square roots proves \eqref{eq:projected_noise_rate}.
\end{proof}

\begin{proposition}[Optimality over finite-dimensional spaces]
  \label{prop:stochastic_approximation_optimality}
  Let $S_N\subset L^2(\Omega)$ be any deterministic subspace with
  $\dim S_N=N$, and let $Z_N\in L^2(\Xi;S_N)$ be arbitrary.  Then
  \begin{equation}
    \mathbb E
    \Norm{
      Y-Z_N
    }_{L^2(\Omega)}^2
    \geq
    \sum_{j>N}
    \lambda_j^{-2}.
    \label{eq:stochastic_projection_spectral_lower}
  \end{equation}
  Equality is attained when
  \[
    S_N
    =
    \operatorname{span}
    \qc{
      \phi_1,\ldots,\phi_N
    }
  \]
  and $Z_N$ is the $L^2$-orthogonal projection of $Y$ onto $S_N$.

  In particular, taking $S_N=V_h$ and $N=N_h$, under
  Assumption~\ref{ass:poly_mesh},
  \begin{equation}
    \Norm{
      Y-Z_h
    }_{L^2\qp{\Xi;L^2(\Omega)}}
    \gtrsim
    h^{2-d/2}
    \qquad
    \forall Z_h\in L^2(\Xi;V_h).
    \label{eq:stochastic_approximation_lower_rate}
  \end{equation}
\end{proposition}

\begin{proof}
  Let $P_N:L^2(\Omega)\to S_N$ be the $L^2$-orthogonal projection.
  Since $Z_N$ takes values in $S_N$,
  \[
    \Norm{
      Y-Z_N
    }_{L^2(\Omega)}^2
    =
    \Norm{
      Y-P_NY
    }_{L^2(\Omega)}^2
    +
    \Norm{
      P_NY-Z_N
    }_{L^2(\Omega)}^2
  \]
  almost surely.  Hence
  \[
    \mathbb E
    \Norm{
      Y-Z_N
    }_{L^2(\Omega)}^2
    \geq
    \mathbb E
    \Norm{
      Y-P_NY
    }_{L^2(\Omega)}^2.
  \]

  The covariance operator of $Y$ is $\Lop^{-2}$, which is trace
  class for $d\in\qc{2,3}$.  Therefore,
  \begin{align*}
    \mathbb E
    \Norm{
      Y-P_NY
    }_{L^2(\Omega)}^2
    &=
    \operatorname{Tr}
    \qp{
      \qp{I-P_N}\Lop^{-2}\qp{I-P_N}
    }
    \\
    &=
    \operatorname{Tr}
    \qp{
      \qp{I-P_N}\Lop^{-2}
    }.
  \end{align*}
  Since $P_N$ has rank $N$, the variational characterisation of the
  eigenvalues of $\Lop^{-2}$ gives
  \[
    \operatorname{Tr}
    \qp{
      \qp{I-P_N}\Lop^{-2}
    }
    \geq
    \sum_{j>N}
    \lambda_j^{-2}.
  \]
  This proves
  \eqref{eq:stochastic_projection_spectral_lower}.  Equality is
  attained by the spectral space
  $\operatorname{span}\qc{\phi_1,\ldots,\phi_N}$.

  Finally, take $S_N=V_h$ and $N=N_h$.  By
  \eqref{eq:weyl},
  \[
    \sum_{j>N_h}
    \lambda_j^{-2}
    \gtrsim
    N_h^{1-4/d}.
  \]
  Since Assumption~\ref{ass:poly_mesh} gives
  $N_h\simeq h^{-d}$,
  \[
    N_h^{1-4/d}
    \simeq
    h^{4-d}.
  \]
  Taking square roots proves
  \eqref{eq:stochastic_approximation_lower_rate}.
\end{proof}

\begin{theorem}[Sharp strong $L^2$ convergence]
  \label{thm:main_spde_convergence}
  Let $d\in\qc{2,3}$ and let $k\ge1$ be fixed, let
  Assumptions~\ref{ass:elliptic_regular} and \ref{ass:poly_mesh} hold
  and let $\gamma$ be sufficiently large as in
  Lemma~\ref{lem:sip_stability}. Then, there exist constants $c,C>0$,
  independent of $h$, such that the solution $Y_h$ of
  \eqref{eq:discrete_spde} satisfies
  \begin{equation}
    \boxed{
      c h^{2-d/2}
      \leq
      \Norm{
        Y-Y_h
      }_{L^2\qp{\Xi;L^2(\Omega)}}
      \leq
      C h^{2-d/2}.
    }
    \label{eq:main_spde_convergence}
  \end{equation}
  Thus the sharp strong $L^2$ rate is $h$ in two dimensions and
  $h^{1/2}$ in three dimensions.
\end{theorem}

\begin{proof}
  The decomposition \eqref{eq:basic_error_split} and
  Lemma~\ref{lem:projected_white_noise} give
  \begin{equation*}
    \Norm{
      Y-Y_h
    }_{L^2\qp{\Xi;L^2(\Omega)}}
    \lesssim
    h^{2-d/2}
    +
    \Norm{
      Y^{(h)}-Y_h
    }_{L^2\qp{\Xi;L^2(\Omega)}}.
  \end{equation*}
  Applying Lemma~\ref{lem:deterministic_L2} samplewise with
  $f=\Wh$, followed by \eqref{eq:white_noise_dimension}, gives
  \begin{align*}
    \Norm{
      Y^{(h)}-Y_h
    }_{L^2\qp{\Xi;L^2(\Omega)}}
    &\lesssim
    h^2
    \Norm{
      \Wh
    }_{L^2\qp{\Xi;L^2(\Omega)}}
    \\
    &\lesssim
    h^{2-d/2}.
  \end{align*}
  This proves the upper bound.  The lower bound follows from
  Proposition~\ref{prop:stochastic_approximation_optimality} with
  $Z_h=Y_h$.
\end{proof}

\section{Extension to a reentrant polygon}
\label{sec:reentrant_polygon}

To examine the impact of geometry, we now replace
Assumption~\ref{ass:elliptic_regular} by the regularity structure
associated with a single reentrant corner.  The deterministic SIPDG
analysis for elliptic problems on polygons with corner singularities
is well established.  In particular, the Dirichlet corner exponent,
low-regularity continuity and consistency of the SIP form, Galerkin
orthogonality, energy quasi-optimality and adjoint-consistent $L^2$
duality are developed in \cite[Remark~2.6; Propositions~5.6--5.7;
  Lemmas~5.12--5.13,
  5.15--5.16]{MullerSchoetzauSchwab2017SIPPolygons}.  We make use of
the associated weighted-space framework.  The points requiring
attention here are instead that the meshes are the star-shaped
polytopes of Assumption~\ref{ass:poly_mesh}, and that the stochastic
analysis requires the endpoint behaviour of the leading corner
singularity.

Throughout this section, the penalty parameter is chosen sufficiently
large for Lemma~\ref{lem:sip_stability} to hold. Let $d=2$ and let
$\Omega$ be a bounded polygon with one reentrant corner $x_\star$ of
interior angle
\begin{equation}
  \omega_\star\in\qp{\pi,2\pi}.
  \label{eq:reentrant_angle}
\end{equation}
All remaining vertices on $\partial\Omega$ are assumed to be convex.
We assume that the meshes are boundary-conforming and that $x_\star$
is a mesh vertex, see Figure \ref{fig:reentrant_corner}.  Since
$\omega_\star\in(\pi,2\pi)$,
\begin{equation}
  \frac{\pi}{\omega_\star}
  \in
  \qp{\frac12,1}.
  \label{eq:corner_exponent}
\end{equation}

\begin{figure}[h!]
  \centering
  \includegraphics[width=0.3\textwidth]{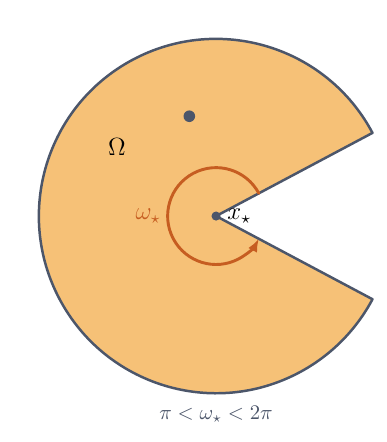}
  \caption{
    Local geometry at the reentrant corner $x_\star$, with interior
    angle $\omega_\star$ and polar coordinates
    $\qp{r_\star,\theta_\star}$.
  }
  \label{fig:reentrant_corner}
\end{figure}

Let $\chi_\star$ be a smooth radial cut-off supported near
$x_\star$ and equal to one in a neighbourhood of $x_\star$, and set
\begin{equation}
  s_\star
  :=
  \chi_\star\qp{r_\star}
  r_\star^{\pi/\omega_\star}
  \sin{
    \frac{\pi\theta_\star}{\omega_\star}
  }.
  \label{eq:corner_singularity}
\end{equation}
Then $s_\star\in H^1_0(\Omega)$,
$\Lop s_\star\in L^2(\Omega)$, and, near $x_\star$,
\begin{equation}
  \abs{
    \nabla^m s_\star
  }
  \lesssim
  r_\star^{\pi/\omega_\star-m},
  \qquad
  m\in\qc{0,1,2}.
  \label{eq:corner_pointwise_behaviour}
\end{equation}

\subsection{Corner structure and deterministic SIP estimates}
\label{subsec:corner_decomposition}

Write
\[
  T:=\Lop^{-1}.
\]
The classical Dirichlet corner decomposition
\cite{Grisvard1985EllipticProblemsNonsmoothDomains,
      Dauge1988EllipticBoundaryValueProblems}
gives bounded maps
\[
  T_{\mathrm{reg}}
  :
  L^2(\Omega)
  \longrightarrow
  H^2(\Omega)\cap H^1_0(\Omega),
  \qquad
  \ell_\star
  \in
  \qp{L^2(\Omega)}^\ast,
\]
such that
\begin{equation}
  Tf
  =
  T_{\mathrm{reg}}f
  +
  \ell_\star(f)s_\star,
  \qquad
  \Norm{
    T_{\mathrm{reg}}f
  }_{H^2(\Omega)}
  +
  \abs{
    \ell_\star(f)
  }
  \lesssim
  \Norm{f}_{L^2(\Omega)}.
  \label{eq:corner_decomposition}
\end{equation}
Let $p_\star\in L^2(\Omega)$ denote the Riesz representative of
$\ell_\star$.

We shall also use
\begin{equation}
  \sigma_j\qp{T_{\mathrm{reg}}}
  \lesssim
  j^{-1},
  \qquad
  j\ge1.
  \label{eq:regular_resolvent_singular_values}
\end{equation}
Indeed, Weyl's law gives
$\sigma_j(T)=\lambda_j^{-1}\lesssim j^{-1}$ in two dimensions,
while $T-T_{\mathrm{reg}}$ has rank one, so
\eqref{eq:regular_resolvent_singular_values} follows from the
finite-rank perturbation inequality for singular values.

Since $\pi/\omega_\star>1/2$, the gradient of $s_\star$ has an
$L^2$ trace on every mesh face.  Hence, for each fixed mesh,
$a_h^\sharp$ extends naturally to
\begin{equation}
  \mathcal V_{h,\star}^\sharp
  :=
  \qp{H^2(\Omega)\cap H^1_0(\Omega)}
  +
  \operatorname{span}\qc{s_\star}
  +
  V_h.
  \label{eq:corner_extended_sip_space}
\end{equation}
The uniform trace estimate required for the approximation analysis is
proved in Lemma~\ref{lem:corner_approximation}.

We may therefore define the SIP Ritz map
\[
  \mathcal G_h:
  \qp{H^2(\Omega)\cap H^1_0(\Omega)}
  +
  \operatorname{span}\qc{s_\star}
  \longrightarrow
  V_h
\]
by
\begin{equation}
  a_h\qp{\mathcal G_hv,v_h}
  =
  a_h^\sharp\qp{v,v_h}
  \qquad
  \forall v_h\in V_h.
  \label{eq:corner_ritz_map}
\end{equation}
The usual consistency argument remains valid.  To see the only
nonstandard point, perform elementwise integration by parts after
removing a circle of radius $r$ about $x_\star$.  On that circle,
\[
  s_\star
  =
  O\qp{r^{\pi/\omega_\star}},
  \qquad
  \partial_ns_\star
  =
  O\qp{r^{\pi/\omega_\star-1}},
\]
and the boundary length is $O(r)$.  The additional contribution
therefore vanishes as $r\downarrow0$.  Consequently,
\begin{equation}
  a_h^\sharp\qp{
    Tf,v_h
  }
  =
  \qp{
    f,v_h
  }_{L^2(\Omega)}
  \qquad
  \forall v_h\in V_h,
  \label{eq:corner_sip_consistency}
\end{equation}
and, with $T_h=L_h^{-1}\Pih$,
\begin{equation}
  T_h
  =
  \mathcal G_hT.
  \label{eq:discrete_resolvent_ritz_identity}
\end{equation}

The following estimates collect the deterministic information needed
for the stochastic analysis.

\begin{lemma}[Corner SIP estimates]
  \label{lem:corner_approximation}
  For fixed polynomial degree $k\ge1$,
  \begin{align}
    \Norm{
      u-\mathcal G_hu
    }_{h,\ast}
    &\lesssim
    h
    \Norm{u}_{H^2(\Omega)},
    &&u\in H^2(\Omega)\cap H^1_0(\Omega),
    \label{eq:regular_ritz_energy}
    \\
    \Norm{
      s_\star-\Pih s_\star
    }_{h,\ast}
    +
    \Norm{
      s_\star-\mathcal G_hs_\star
    }_{h,\ast}
    &\lesssim
    h^{\pi/\omega_\star},
    \label{eq:corner_ritz_energy}
    \\
    \Norm{
      s_\star-\mathcal G_hs_\star
    }_{L^2(\Omega)}
    &\lesssim
    h^{2\pi/\omega_\star}.
    \label{eq:corner_ritz_L2}
  \end{align}
\end{lemma}

\begin{proof}
  The estimate \eqref{eq:regular_ritz_energy} follows from
  \eqref{eq:extended_projection_approximation}, coercivity,
  continuity and Galerkin orthogonality.  We therefore concentrate on
  the singular function $s_\star$.

  Introduce the metric corner patch
  \begin{equation}
    \mathcal N_h^\star
    :=
    \ensemble{
      K\in\mathcal T_h
    }{
      K\cap B\qp{x_\star,h}\neq\emptyset
    }.
    \label{eq:metric_corner_patch}
  \end{equation}
  If $K\in\mathcal N_h^\star$, then
  $K\subset B(x_\star,2h)$.  Moreover,
  Assumption~\ref{ass:poly_mesh},
  \eqref{eq:element_volume_regular} and
  \eqref{eq:quasi_uniform} give
  $\abs{K}\gtrsim h^2$.  Since the element interiors are disjoint,
  \begin{equation}
    \#\mathcal N_h^\star
    \lesssim
    1.
    \label{eq:metric_corner_patch_cardinality}
  \end{equation}
  On the other hand,
  \begin{equation}
    K\notin\mathcal N_h^\star
    \quad\Longrightarrow\quad
    \operatorname{dist}\qp{K,x_\star}
    \geq
    h.
    \label{eq:corner_patch_separation}
  \end{equation}
  These two properties separate the corner contribution from the
  regular part of the mesh.

  We first estimate the contribution from
  $\mathcal N_h^\star$.  Quasi-uniformity implies that, for
  $K\in\mathcal N_h^\star$,
  \[
    K
    \subset
    B\qp{x_\star,Ch_K}
  \]
  with $C$ independent of $h$.  Hence
  \eqref{eq:corner_pointwise_behaviour} gives
  \begin{align}
    \Norm{
      s_\star
    }_{L^2(K)}
    &\lesssim
    h_K^{1+\pi/\omega_\star},
    &
    \Norm{
      \nabla s_\star
    }_{L^2(K)}
    &\lesssim
    h_K^{\pi/\omega_\star}.
    \label{eq:corner_volume_scaling}
  \end{align}

  The only nonstandard estimate is the whole-boundary trace of
  $\nabla s_\star$.  Since
  $\pi/\omega_\star>1/2$, choose an exponent $\mu$ such that
  \begin{equation}
    \frac43
    <
    \mu
    <
    \frac{
      2
    }{
      2-\pi/\omega_\star
    },
    \qquad
    \mu'
    :=
    \frac{\mu}{\mu-1}.
    \label{eq:corner_trace_exponent}
  \end{equation}
  Such a choice is possible precisely because
  $\pi/\omega_\star>1/2$.  The pointwise behaviour
  \eqref{eq:corner_pointwise_behaviour} then gives
  \begin{align}
    \Norm{
      \nabla s_\star
    }_{L^{\mu'}(K)}
    &\lesssim
    h_K^{
      \pi/\omega_\star-1+2/\mu'
    },
    \label{eq:corner_gradient_mup}
    \\
    \Norm{
      D^2s_\star
    }_{L^\mu(K)}
    &\lesssim
    h_K^{
      \pi/\omega_\star-2+2/\mu
    }.
    \label{eq:corner_hessian_mu}
  \end{align}
  In particular,
  $\nabla s_\star\in[W^{1,\mu}(K)]^2$.

  Apply \cite[Lemma~2.2]{botti2025trace} with $p=\mu$ and
  $q=2$ to each component of $\nabla s_\star$.  The mesh condition
  required there holds uniformly by
  \eqref{eq:botti_geometric_condition}.  Since
  $\abs{K}\simeq h_K^2$, we obtain
  \begin{align*}
    \Norm{
      \nabla s_\star
    }_{L^2(\partial K)}^2
    &\lesssim
    h_K^{-1}
    \abs{K}^{1-2/\mu'}
    \Norm{
      \nabla s_\star
    }_{L^{\mu'}(K)}^2
    \\
    &\qquad
    +
    \Norm{
      D^2s_\star
    }_{L^\mu(K)}
    \Norm{
      \nabla s_\star
    }_{L^{\mu'}(K)}
    \\
    &\lesssim
    h_K^{2\pi/\omega_\star-1}.
  \end{align*}
  Thus
  \begin{equation}
    \Norm{
      \nabla s_\star
    }_{L^2(\partial K)}
    \lesssim
    h_K^{\pi/\omega_\star-1/2}.
    \label{eq:corner_whole_boundary_gradient_trace}
  \end{equation}
  Notice that this is a whole-boundary estimate; no lower bound on
  the size of an individual face is used.

  Applying \eqref{eq:poly_trace} directly to $s_\star$ and using
  \eqref{eq:corner_volume_scaling} similarly gives
  \begin{equation}
    \Norm{
      s_\star
    }_{L^2(\partial K)}
    \lesssim
    h_K^{\pi/\omega_\star+1/2}.
    \label{eq:corner_whole_boundary_value_trace}
  \end{equation}

  Let $\Pi_K$ denote the local $L^2$ projection.  Its $L^2$
  stability, together with the polynomial inverse and trace estimates
  of Lemma~\ref{lem:poly_trace_inverse}, now yields
  \begin{equation}
    \begin{split}
      &
      \Norm{
        s_\star-\Pi_Ks_\star
      }_{L^2(K)}
      +
      h_K
      \Norm{
        \nabla\qp{s_\star-\Pi_Ks_\star}
      }_{L^2(K)}
      \\
      &\qquad
      +
      h_K^{1/2}
      \Norm{
        s_\star-\Pi_Ks_\star
      }_{L^2(\partial K)}
      +
      h_K^{3/2}
      \Norm{
        \nabla\qp{s_\star-\Pi_Ks_\star}
      }_{L^2(\partial K)}
      \lesssim
      h_K^{1+\pi/\omega_\star}
    \end{split}
    \label{eq:corner_patch_scaling}
  \end{equation}
  for every $K\in\mathcal N_h^\star$.

  We next consider the elements outside the metric corner patch.
  By \eqref{eq:corner_patch_separation}, the singularity is separated
  from each such element by at least $h$.  Hence
  \eqref{eq:poly_projection_approximation} and
  \eqref{eq:corner_pointwise_behaviour} give, for a fixed radius
  $R>0$ containing the support of $\chi_\star$,
  \begin{align*}
    \sum_{K\notin\mathcal N_h^\star}
    h_K^2
    \Norm{
      s_\star
    }_{H^2(K)}^2
    &\lesssim
    h^2
    \left(
      1
      +
      \int_h^R
      r^{2\pi/\omega_\star-3}
      \,\mathrm dr
    \right)
    \\
    &\lesssim
    h^{2\pi/\omega_\star}.
  \end{align*}
  Here the fixed term accounts for the smooth cut-off region.
  Combining this estimate with
  \eqref{eq:corner_patch_scaling},
  \eqref{eq:metric_corner_patch_cardinality} and the definition of
  $\Norm{\cdot}_{h,\ast}$ gives
  \begin{equation}
    \Norm{
      s_\star-\Pih s_\star
    }_{h,\ast}
    \lesssim
    h^{\pi/\omega_\star}.
    \label{eq:corner_projection_energy}
  \end{equation}

  To obtain the corresponding Ritz estimate, set
  \[
    \eta
    :=
    s_\star-\Pih s_\star,
    \qquad
    \xi_h
    :=
    \Pih s_\star-\mathcal G_hs_\star.
  \]
  Consistency and Galerkin orthogonality imply
  \[
    a_h\qp{
      \xi_h,v_h
    }
    =
    -
    a_h^\sharp\qp{
      \eta,v_h
    }
    \qquad
    \forall v_h\in V_h.
  \]
  Taking $v_h=\xi_h$ and using coercivity,
  \eqref{eq:extended_sip_continuity} and
  \eqref{eq:discrete_extended_norm_control} gives
  \[
    \Norm{
      \xi_h
    }_{h,\ast}
    \lesssim
    \Norm{
      \eta
    }_{h,\ast}.
  \]
  Together with \eqref{eq:corner_projection_energy}, this proves
  \eqref{eq:corner_ritz_energy}.

  It remains to prove the $L^2$ estimate.  Set
  \[
    e_\star
    :=
    s_\star-\mathcal G_hs_\star,
    \qquad
    \psi
    :=
    Te_\star
    =
    T_{\mathrm{reg}}e_\star
    +
    \ell_\star(e_\star)s_\star,
  \]
  and choose
  \[
    \psi_h
    :=
    \Pih T_{\mathrm{reg}}e_\star
    +
    \ell_\star(e_\star)\Pih s_\star.
  \]
  The regular approximation estimate,
  \eqref{eq:corner_projection_energy} and
  \eqref{eq:corner_decomposition} imply
  \begin{align*}
    \Norm{
      \psi-\psi_h
    }_{h,\ast}
    &\lesssim
    h
    \Norm{
      T_{\mathrm{reg}}e_\star
    }_{H^2(\Omega)}
    +
    h^{\pi/\omega_\star}
    \abs{
      \ell_\star(e_\star)
    }
    \\
    &\lesssim
    h^{\pi/\omega_\star}
    \Norm{
      e_\star
    }_{L^2(\Omega)},
  \end{align*}
  where we used $h\lesssim h^{\pi/\omega_\star}$ for $0<h\le1$.

  Finally, symmetry, consistency and Galerkin orthogonality give
  \begin{align*}
    \Norm{
      e_\star
    }_{L^2(\Omega)}^2
    &=
    a_h^\sharp\qp{
      e_\star,\psi-\psi_h
    }
    \\
    &\lesssim
    \Norm{
      e_\star
    }_{h,\ast}
    \Norm{
      \psi-\psi_h
    }_{h,\ast}
    \\
    &\lesssim
    h^{2\pi/\omega_\star}
    \Norm{
      e_\star
    }_{L^2(\Omega)}.
  \end{align*}
  This proves \eqref{eq:corner_ritz_L2}.
\end{proof}

\subsection{Hilbert--Schmidt resolvent error}
\label{subsec:corner_hilbert_schmidt}

The preceding estimates still reflect the deterministic corner loss.
For the SPDE, however, the singular contribution in
\eqref{eq:corner_decomposition} is only rank one.  This allows the
regular infinite-dimensional part and the corner correction to be
estimated separately.

\begin{lemma}[Reentrant-corner resolvent estimate]
  \label{lem:corner_hilbert_schmidt}
  The continuous and discrete resolvents satisfy
  \begin{equation}
    \Norm{
      T-T_h
    }_{\mathcal L_2\qp{L^2(\Omega)}}
    \lesssim
    h.
    \label{eq:corner_hilbert_schmidt}
  \end{equation}
\end{lemma}

\begin{proof}
  Let $u\in H^2(\Omega)\cap H^1_0(\Omega)$ and set
  \[
    e_u:=u-\mathcal G_hu.
  \]
  For $g\in L^2(\Omega)$, adjoint consistency gives
  \[
    \qp{
      e_u,g
    }_{L^2(\Omega)}
    =
    a_h^\sharp\qp{
      e_u,Tg
    }.
  \]
  Using the corner decomposition
  \eqref{eq:corner_decomposition} and the Galerkin orthogonality
  \[
    a_h^\sharp\qp{
      e_u,v_h
    }
    =
    0
    \qquad
    \forall v_h\in V_h,
  \]
  we obtain
  \[
    \qp{
      e_u,g
    }_{L^2(\Omega)}
    =
    a_h^\sharp\qp{
      e_u,
      T_{\mathrm{reg}}g-\Pih T_{\mathrm{reg}}g
    }
    +
    \ell_\star(g)
    a_h^\sharp\qp{
      e_u,
      s_\star-\Pih s_\star
    }.
  \]

  The first term on the right is a bounded linear functional of $g$.
  Hence, by the Riesz representation theorem, there exists a unique
  $B_hu\in L^2(\Omega)$ such that
  \[
    \qp{
      B_hu,g
    }_{L^2(\Omega)}
    :=
    a_h^\sharp\qp{
      e_u,
      T_{\mathrm{reg}}g-\Pih T_{\mathrm{reg}}g
    }.
  \]
  Set also
  \[
    q_h(u)
    :=
    a_h^\sharp\qp{
      e_u,
      s_\star-\Pih s_\star
    }.
  \]
  Lemma~\ref{lem:corner_approximation} and
  \eqref{eq:corner_decomposition} give
  \begin{equation}
    \Norm{
      B_hu
    }_{L^2(\Omega)}
    \lesssim
    h^2
    \Norm{u}_{H^2(\Omega)},
    \qquad
    \abs{
      q_h(u)
    }
    \lesssim
    h^{1+\pi/\omega_\star}
    \Norm{u}_{H^2(\Omega)}.
    \label{eq:corner_regular_error_bounds}
  \end{equation}

  Let $p_\star\in L^2(\Omega)$ be the Riesz representative of
  $\ell_\star$, so that
  \[
    \ell_\star(g)
    =
    \qp{
      g,p_\star
    }_{L^2(\Omega)}.
  \]
  Since the preceding identity holds for every $g\in L^2(\Omega)$,
  \[
    u-\mathcal G_hu
    =
    B_hu
    +
    q_h(u)p_\star.
  \]
  Applying this with $u=T_{\mathrm{reg}}f$ and using
  \eqref{eq:corner_decomposition} together with
  \eqref{eq:discrete_resolvent_ritz_identity}, we obtain
  \begin{equation}
    \qp{T-T_h}f
    =
    B_hT_{\mathrm{reg}}f
    +
    q_h\qp{T_{\mathrm{reg}}f}p_\star
    +
    \ell_\star(f)
    \qp{
      s_\star-\mathcal G_hs_\star
    }.
    \label{eq:corner_resolvent_operator_decomposition}
  \end{equation}

  We first consider the infinite-dimensional term
  $B_hT_{\mathrm{reg}}$.  By
  \eqref{eq:corner_regular_error_bounds},
  \[
    \Norm{
      B_hT_{\mathrm{reg}}
    }_{\mathcal L\qp{L^2(\Omega)}}
    \lesssim
    h^2.
  \]
  From
  \[
    u-\mathcal G_hu
    =
    B_hu+q_h(u)p_\star
  \]
  we have, for every $f\in L^2(\Omega)$,
  \[
    \qp{
      B_hT_{\mathrm{reg}}-T_{\mathrm{reg}}
    }f
    =
    -\mathcal G_hT_{\mathrm{reg}}f
    -
    q_h\qp{T_{\mathrm{reg}}f}p_\star.
  \]
  The first term takes values in $V_h$, while the second takes values
  in $\operatorname{span}\qc{p_\star}$.  Consequently,
  \[
    \operatorname{rank}
    \qp{
      B_hT_{\mathrm{reg}}-T_{\mathrm{reg}}
    }
    \leq
    N_h+1.
  \]
    We use the standard finite-rank perturbation inequality for singular
  values: if $A$ and $B$ are compact operators and
  \[
    \operatorname{rank}\qp{A-B}\leq r,
  \]
  then
  \[
    \sigma_{r+j}\qp{A}
    \leq
    \sigma_j\qp{B},
    \qquad
    j\geq1.
  \]
  Applying this with
  \[
    A=B_hT_{\mathrm{reg}},
    \qquad
    B=T_{\mathrm{reg}},
    \qquad
    r=N_h+1,
  \]
  and using \eqref{eq:regular_resolvent_singular_values}, we obtain
  \[
    \sigma_{N_h+1+j}
    \qp{
      B_hT_{\mathrm{reg}}
    }
    \leq
    \sigma_j\qp{T_{\mathrm{reg}}}
    \lesssim
    j^{-1},
    \qquad
    j\geq1.
  \]
  Combining this with the operator-norm bound above and
  $N_h\simeq h^{-2}$ yields
  \begin{align}
    \Norm{
      B_hT_{\mathrm{reg}}
    }_{\mathcal L_2\qp{L^2(\Omega)}}^2
    &\lesssim
    N_hh^4
    +
    \sum_{j=1}^\infty
    \min\qc{
      h^4,j^{-2}
    }
    \lesssim
    h^2.
    \label{eq:BhTreg_hilbert_schmidt}
  \end{align}
  Hence
  \[
    \Norm{
      B_hT_{\mathrm{reg}}
    }_{\mathcal L_2\qp{L^2(\Omega)}}
    \lesssim
    h.
  \]

  The remaining two terms in
  \eqref{eq:corner_resolvent_operator_decomposition} have rank one.
  Using \eqref{eq:corner_regular_error_bounds},
  \eqref{eq:corner_decomposition} and
  \eqref{eq:corner_ritz_L2}, we obtain
  \[
    \Norm{
      f\longmapsto
      q_h\qp{T_{\mathrm{reg}}f}p_\star
    }_{\mathcal L_2\qp{L^2(\Omega)}}
    \lesssim
    h^{1+\pi/\omega_\star},
  \]
  and
  \[
    \Norm{
      f\longmapsto
      \ell_\star(f)
      \qp{
        s_\star-\mathcal G_hs_\star
      }
    }_{\mathcal L_2\qp{L^2(\Omega)}}
    \lesssim
    h^{2\pi/\omega_\star}.
  \]
  Since
  \[
    \frac{\pi}{\omega_\star}>\frac12,
  \]
  both exponents are strictly larger than one.  Thus the two
  rank-one corner terms are of higher order than the $O(h)$
  Hilbert--Schmidt contribution from $B_hT_{\mathrm{reg}}$, and
  \eqref{eq:corner_hilbert_schmidt} follows.
\end{proof}

\begin{theorem}[Sharp convergence on a reentrant polygon]
  \label{thm:reentrant_spde_convergence}
  Let $d=2$ and let $\Omega$ have the corner geometry described
  above.  Under Assumption~\ref{ass:poly_mesh}, there exist constants
  $c,C>0$, independent of $h$, such that
  \begin{equation}
    \boxed{
      ch
      \leq
      \Norm{
        Y-Y_h
      }_{L^2\qp{\Xi;L^2(\Omega)}}
      \leq
      Ch.
    }
    \label{eq:reentrant_spde_convergence}
  \end{equation}
\end{theorem}

\begin{proof}
  Let $\qc{e_j}_{j\ge1}$ be any $L^2(\Omega)$-orthonormal basis and
  set $\zeta_j:=\mathcal W(e_j)$.  Since the continuous and discrete
  problems are driven by the same isonormal process,
  \[
    Y
    =
    \sum_{j=1}^\infty
    \zeta_jTe_j,
    \qquad
    Y_h
    =
    \sum_{j=1}^\infty
    \zeta_jT_he_j,
  \]
  and therefore
  \begin{equation}
    \Norm{
      Y-Y_h
    }_{L^2\qp{\Xi;L^2(\Omega)}}
    =
    \Norm{
      T-T_h
    }_{\mathcal L_2\qp{L^2(\Omega)}}.
    \label{eq:strong_error_hilbert_schmidt}
  \end{equation}
  Lemma~\ref{lem:corner_hilbert_schmidt} gives the upper bound.
  The lower bound follows from
  Proposition~\ref{prop:stochastic_approximation_optimality}, since
  $Y_h$ takes values in $V_h$ and $N_h\simeq h^{-2}$.
\end{proof}

\begin{remark}[Why the corner does not change the stochastic rate]
  \label{rem:corner_rate_mechanism}
  A global deterministic regularity argument only sees
  $H^{1+\pi/\omega_\star-\varepsilon}(\Omega)$ regularity and hence
  reflects the loss caused by the reentrant corner.  The operator
  decomposition
  \eqref{eq:corner_resolvent_operator_decomposition} is sharper.
  The infinite-dimensional regular contribution has
  Hilbert--Schmidt error $O(h)$, whereas the two corner corrections
  are rank one and have orders
  \[
    O\qp{
      h^{1+\pi/\omega_\star}
    }
    \qquad\text{and}\qquad
    O\qp{
      h^{2\pi/\omega_\star}
    }.
  \]
  Both are $o(h)$ because $\omega_\star<2\pi$.  Thus the reentrant
  corner degrades deterministic approximation rates without changing
  the sharp first-order strong rate of the white-noise problem.
\end{remark}

\section{Numerical experiments}
\label{sec:numerics}

The numerical experiments examine the two main convergence results.
On the unit square and cube, the explicit Dirichlet eigensystem is
used to estimate the continuum strong error and test the
dimension-dependent rate in
Theorem~\ref{thm:main_spde_convergence}.  We then consider a family
of reentrant polygonal domains with varying corner angle and test the
first-order stochastic rate of
Theorem~\ref{thm:reentrant_spde_convergence} using a coupled
fine-level reference.

The numerical implementation is based on the polytopic discontinuous
Galerkin finite element library \cite{evans2026reyna}. Throughout,
$\kappa=4$ and $V_h$ consists of discontinuous piecewise-linear
polynomials.  We take $\gamma=14$ in two dimensions and $\gamma=24$ in
three dimensions.  Linear systems are solved by ILU-preconditioned
BiCGSTAB with relative tolerance $2\times10^{-9}$.  For each mesh, the
SIP matrix and preconditioner are assembled once and reused for all
modal right-hand sides or stochastic realisations.

\subsection{Continuum strong-error estimator on $(0,1)^d$}
\label{subsec:numerics_setup}

For $\Omega=(0,1)^d$, the eigensystem of $\Lop$ is
\[
  \phi_{\boldsymbol n}(x)
  =
  2^{d/2}
  \prod_{i=1}^d
  \sin{\pi n_i x_i},
  \qquad
  \lambda_{\boldsymbol n}
  =
  \pi^2\abs{\boldsymbol n}_2^2+\kappa^2,
  \qquad
  \boldsymbol n\in\mathbb N^d.
\]
For each mode, let $u_{h,\boldsymbol n}\in V_h$ satisfy
\[
  a_h\qp{
    u_{h,\boldsymbol n},v_h
  }
  =
  \qp{
    \phi_{\boldsymbol n},v_h
  }_{L^2(\Omega)}
  \qquad
  \forall v_h\in V_h.
\]
Expanding the continuous and discrete random fields with respect to
the same isonormal process gives
\begin{equation}
  \mathbb E
  \Norm{
    Y-Y_h
  }_{L^2(\Omega)}^2
  =
  \sum_{\boldsymbol n\in\mathbb N^d}
  e_{h,\boldsymbol n}^2,
  \qquad
  e_{h,\boldsymbol n}^2
  :=
  \Norm{
    \lambda_{\boldsymbol n}^{-1}\phi_{\boldsymbol n}
    -
    u_{h,\boldsymbol n}
  }_{L^2(\Omega)}^2.
  \label{eq:numerical_continuum_mse}
\end{equation}

We evaluate the low-frequency part of
\eqref{eq:numerical_continuum_mse} explicitly and estimate the
remaining infinite sum by importance sampling.  Write
\[
  m\qp{\boldsymbol n}
  :=
  \abs{\boldsymbol n}_\infty.
\]
For a hierarchy with mesh scales $h_1,\ldots,h_L$, the tail proposal
is the equal mixture
\[
  q\qp{\boldsymbol n}
  :=
  \frac1L
  \sum_{\ell=1}^L
  q_\ell\qp{\boldsymbol n},
  \qquad
  q_\ell\qp{\boldsymbol n}
  \propto
  \min\qc{
    h_\ell^4,
    m\qp{\boldsymbol n}^{-4}
  },
  \qquad
  m\qp{\boldsymbol n}>m_0.
\]
Each component resolves the transition
$m(\boldsymbol n)\simeq h_\ell^{-1}$ and retains an infinite
$m^{-4}$ tail.  For independent samples
$\boldsymbol N_q\sim q$, we use
\begin{equation}
  \widehat{\mathcal E}_h^2
  :=
  \sum_{
    m(\boldsymbol n)\leq m_0
  }
  e_{h,\boldsymbol n}^2
  +
  \frac1{N_{\mathrm{tail}}}
  \sum_{q=1}^{N_{\mathrm{tail}}}
  \frac{
    e_{h,\boldsymbol N_q}^2
  }{
    q\qp{\boldsymbol N_q}
  }.
  \label{eq:numerical_modal_estimator}
\end{equation}
The same sampled modes are used across every mesh level and mesh
family in a fixed dimension.

In two dimensions we take $m_0=3$ and
$N_{\mathrm{tail}}=32$, while in three dimensions we take $m_0=2$
and $N_{\mathrm{tail}}=12$.  Error bars show one bootstrap standard
error of the RMS estimator, obtained from $2000$ resamples of the
importance-sampled tail contributions with the explicitly summed
low-frequency contribution held fixed.

For the convergence plots we use $N_K^{-1/d}$ as the mesh scale, where
$N_K$ is the number of computational elements.  Under
Assumption~\ref{ass:poly_mesh}, $N_K^{-1/d}\simeq h$.

\subsection{Experiment 1: regular polytopic meshes}
\label{subsec:numerics_convergence_2d}

\subsubsection{Two dimensions}

For $d=2$, Theorem~\ref{thm:main_spde_convergence} gives the sharp
rate $O(h)$.  We consider four quasi-uniform mesh families: Voronoi,
wave-distorted Voronoi, checkerboard-distorted Voronoi and jittered
quadrilateral meshes.  Each hierarchy contains six refinement levels,
with $N_K=128^2$ elements on the finest level.

Figure~\ref{fig:random-fields-2d} shows representative stochastic
fields on the four mesh families.

\begin{figure}[htbp]
  \centering

  \includegraphics[width=0.48\textwidth]
    {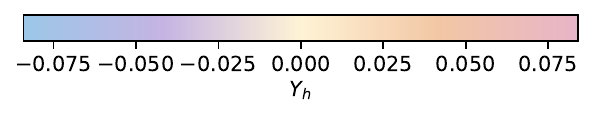}

  \medskip

  \begin{subfigure}[t]{0.24\textwidth}
    \centering
    \includegraphics[width=\linewidth]
      {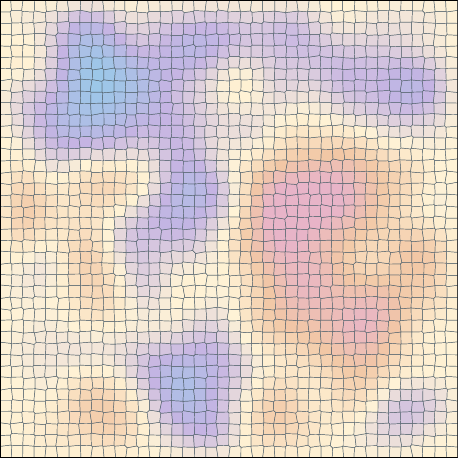}
    \caption{Voronoi}
  \end{subfigure}
  \hfill
  \begin{subfigure}[t]{0.24\textwidth}
    \centering
    \includegraphics[width=\linewidth]
      {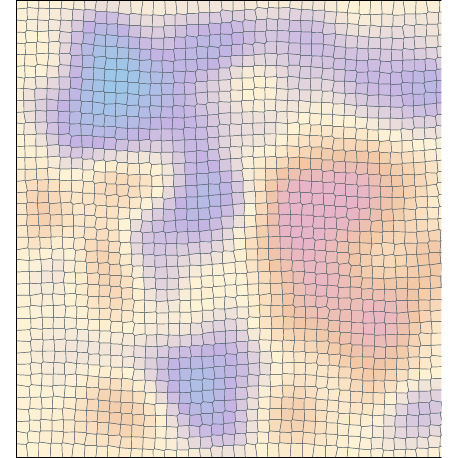}
    \caption{Wave Voronoi}
  \end{subfigure}
  \hfill
  \begin{subfigure}[t]{0.24\textwidth}
    \centering
    \includegraphics[width=\linewidth]
      {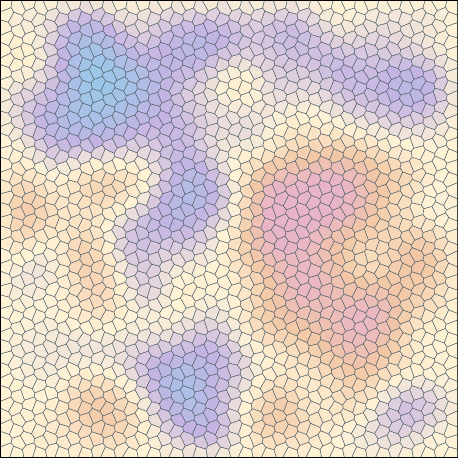}
    \caption{Checkerboard Voronoi}
  \end{subfigure}
  \hfill
  \begin{subfigure}[t]{0.24\textwidth}
    \centering
    \includegraphics[width=\linewidth]
      {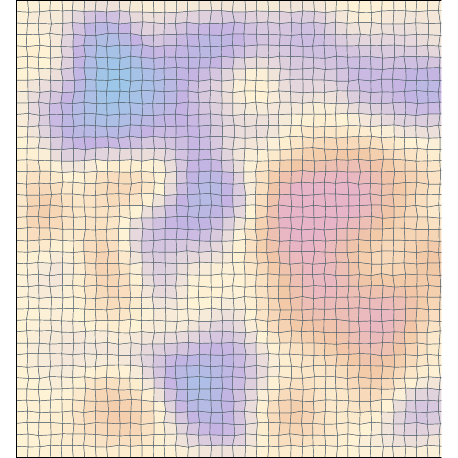}
    \caption{Jittered quadrilaterals}
  \end{subfigure}

  \caption{
    Representative two-dimensional stochastic fields on the four
    quasi-uniform polytopic mesh families.
  }
  \label{fig:random-fields-2d}
\end{figure}

Figure~\ref{fig:convergence-2d} shows the continuum RMS error
computed from \eqref{eq:numerical_modal_estimator}, plotted against
$N_K^{-1/2}$.  The dotted line shows the first-order reference slope.

\begin{figure}[htbp]
  \centering
  \includegraphics[width=0.5\textwidth]
    {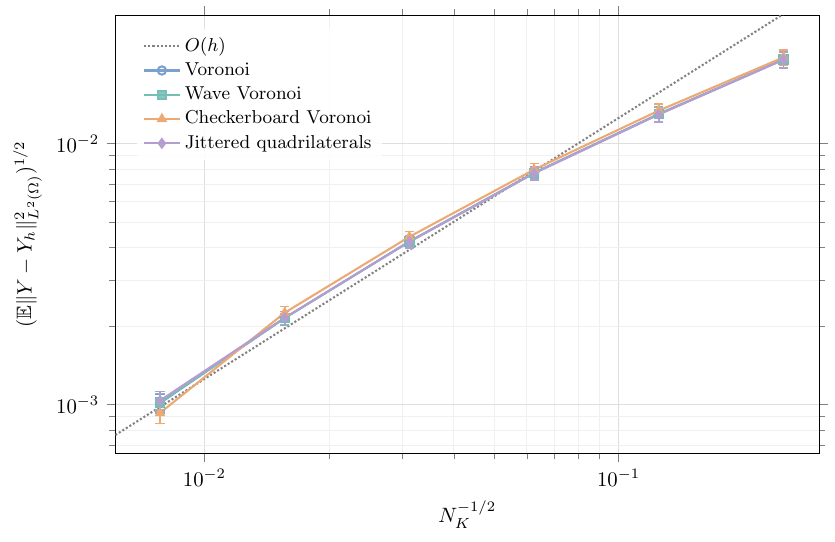}
  \caption{
    Continuum RMS $L^2$ error on the four two-dimensional polytopic
    mesh families.  The horizontal axis is $N_K^{-1/2}$ and the
    dotted line shows the sharp rate $O(h)$ from
    Theorem~\ref{thm:main_spde_convergence}.  Error bars show one
    bootstrap standard error of the modal estimator.
  }
  \label{fig:convergence-2d}
\end{figure}

\subsubsection{Three dimensions}
\label{subsec:numerics_convergence_3d}

For $d=3$, Theorem~\ref{thm:main_spde_convergence} gives the sharp
rate $O(h^{1/2})$.  We consider regular hexahedral, jittered
hexahedral and non-convex polyhedral mesh families.  Each hierarchy
contains five refinement levels.  The regular and jittered families
have $N_K=32^3$ elements on the finest level, while the non-convex
family is generated from a $48^3$ background partition.

The non-convex computational elements are connected agglomerates of
tetrahedra, with a single affine polynomial defined on each resulting
polyhedron.  Figure~\ref{fig:nonconvex-mesh} illustrates the
computational geometry.

\begin{figure}[htbp]
  \centering
  \begin{subfigure}[t]{0.5\textwidth}
    \centering
    \includegraphics[width=\linewidth]
      {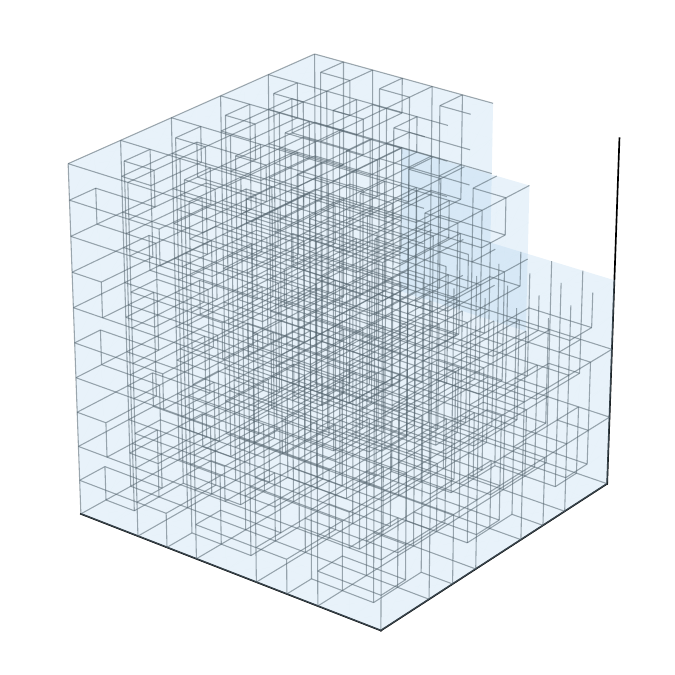}
    \caption{Cut-away polyhedral mesh}
  \end{subfigure}
  \hfill
  \begin{subfigure}[t]{0.35\textwidth}
    \centering
    \includegraphics[width=\linewidth]
      {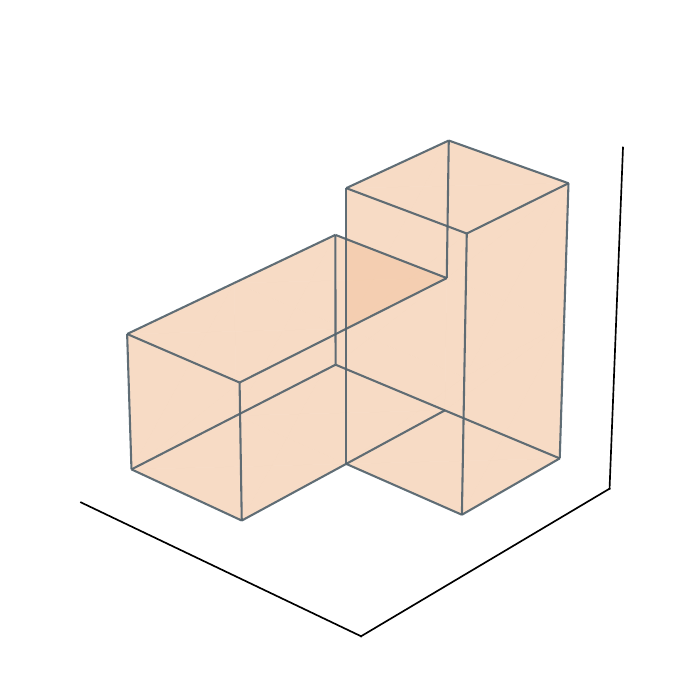}
    \caption{Representative non-convex cell}
  \end{subfigure}
  \caption{
    Non-convex polyhedral geometry used in the three-dimensional
    experiment.  The displayed boundaries are those of the
    computational polyhedra.
  }
  \label{fig:nonconvex-mesh}
\end{figure}

Figure~\ref{fig:convergence-3d} shows the continuum RMS error computed
from \eqref{eq:numerical_modal_estimator}, plotted against
$N_K^{-1/3}$.  The dotted line shows the $O(h^{1/2})$ reference
slope.

\begin{figure}[htbp]
  \centering
  \includegraphics[width=0.5\textwidth]
    {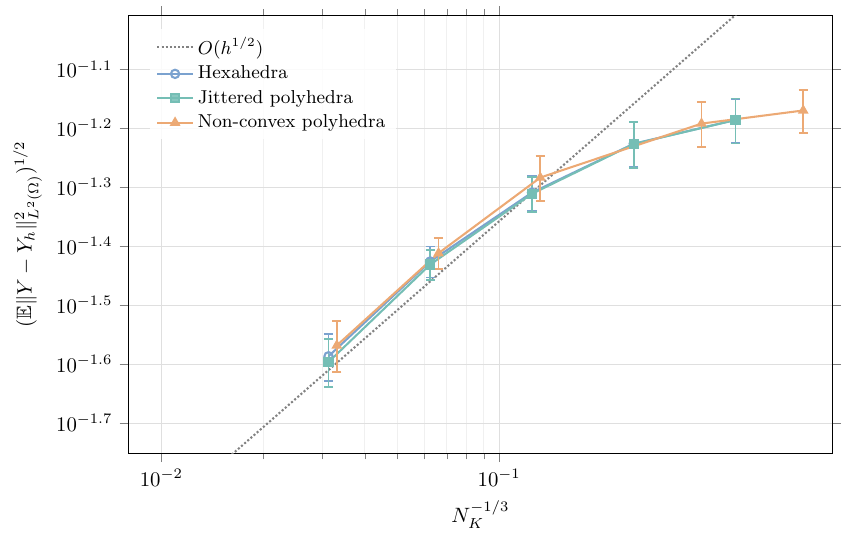}
  \caption{
    Continuum RMS $L^2$ error on the regular hexahedral, jittered
    hexahedral and non-convex polyhedral mesh families.  The
    horizontal axis is $N_K^{-1/3}$ and the dotted line shows the
    sharp rate $O(h^{1/2})$ from
    Theorem~\ref{thm:main_spde_convergence}.  Error bars show one
    bootstrap standard error of the modal estimator.
  }
  \label{fig:convergence-3d}
\end{figure}

\subsection{Experiment 2: stochastic reentrant problem}
\label{subsec:numerics_reentrant}

We consider a family of polygonal domains with
\begin{equation}
  \alpha
  \in
  \qc{
    45^\circ,
    60^\circ,
    75^\circ,
    90^\circ,
    105^\circ,
    120^\circ,
    135^\circ
  }.
  \label{eq:pacman_angles}
\end{equation}
The corresponding reentrant angle is
\begin{equation}
  \omega_\star
  =
  2\pi-\alpha,
  \label{eq:pacman_reentrant_angle}
\end{equation}
so $\pi/\omega_\star$ varies from $4/7$ to $4/5$ across the angle
sweep.

The reentrant vertex is preserved exactly at every refinement level.
The meshes are constructed from a perturbed Cartesian background, cut
to the polygonal domain and agglomerated into general polygons.  The
resulting hierarchies contain many-sided and non-convex elements.
Figure~\ref{fig:pacman-random-fields} shows representative stochastic
fields for four openings.

\begin{figure}[htbp]
  \centering

  \includegraphics[width=0.48\textwidth]
    {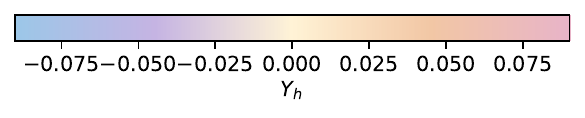}

  \medskip

  \begin{subfigure}[t]{0.24\textwidth}
    \centering
    \includegraphics[width=\linewidth]
      {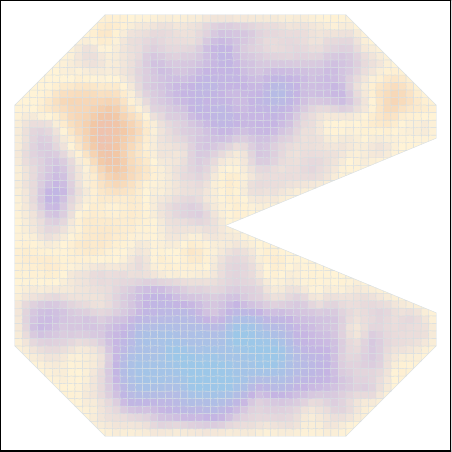}
    \caption{$\alpha=45^\circ$}
  \end{subfigure}
  \hfill
  \begin{subfigure}[t]{0.24\textwidth}
    \centering
    \includegraphics[width=\linewidth]
      {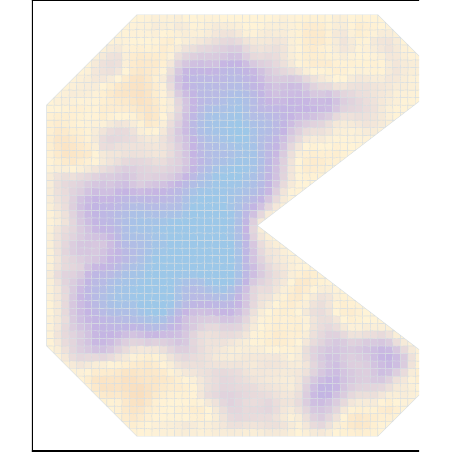}
    \caption{$\alpha=75^\circ$}
  \end{subfigure}
  \hfill
  \begin{subfigure}[t]{0.24\textwidth}
    \centering
    \includegraphics[width=\linewidth]
      {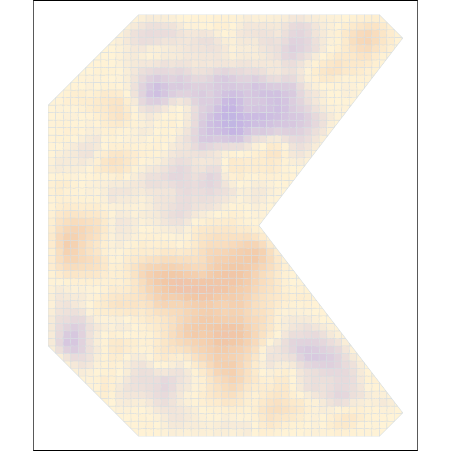}
    \caption{$\alpha=105^\circ$}
  \end{subfigure}
  \hfill
  \begin{subfigure}[t]{0.24\textwidth}
    \centering
    \includegraphics[width=\linewidth]
      {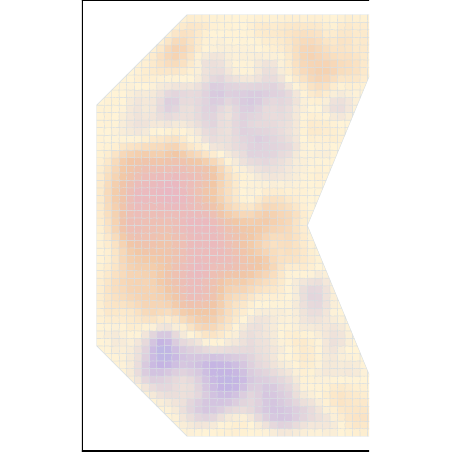}
    \caption{$\alpha=135^\circ$}
  \end{subfigure}

  \caption{
    Representative stochastic fields on the reentrant polygonal
    domains for four mouth openings spanning the angle sweep.
  }
  \label{fig:pacman-random-fields}
\end{figure}

For every mouth opening, the stochastic hierarchy contains five
reported levels followed by two additional refinement levels.  The
final level is generated from a $256\times256$ background partition
and is used as the coupled reference.  For each angle we use $M=24$
independent white-noise realisations.

For each realisation, the same continuum white-noise functional is
restricted to every discrete space in the hierarchy.  Thus the
solutions $Y_h^{(q)}$ and $Y_{h_{\mathrm{ref}}}^{(q)}$ are driven by
the same Gaussian realisation.  On each reported level we compute the
coupled fine-reference RMS difference
\begin{equation}
  \widehat E_{h,\mathrm{ref}}
  :=
  \left(
    \frac1M
    \sum_{q=1}^M
    \Norm{
      Y_h^{(q)}
      -
      Y_{h_{\mathrm{ref}}}^{(q)}
    }_{L^2(\Omega)}^2
  \right)^{1/2}.
  \label{eq:numerical_reference_estimator}
\end{equation}

For each angle, the stochastic convergence rate is estimated from a
least-squares fit of $\log\widehat E_{h,\mathrm{ref}}$ against $\log
N_K^{-1/2}$, omitting the coarsest reported level.  Its sampling
uncertainty is estimated by resampling the $M$ coupled realisations
simultaneously across all mesh levels and repeating the slope fit.  We
use $2000$ paired bootstrap resamples.

Figure~\ref{fig:pacman_stochastic} reports the resulting stochastic
convergence.  Panel~(A) shows the coupled fine-reference RMS
differences for representative $\alpha$, plotted against $N_K^{-1/2}$
together with a first-order reference slope.  Panel~(B) shows the
estimated exponent across the complete angle sweep.

\begin{figure}[htbp]
  \centering
  \begin{subfigure}[t]{0.49\textwidth}
    \centering
    \includegraphics[width=\linewidth]
      {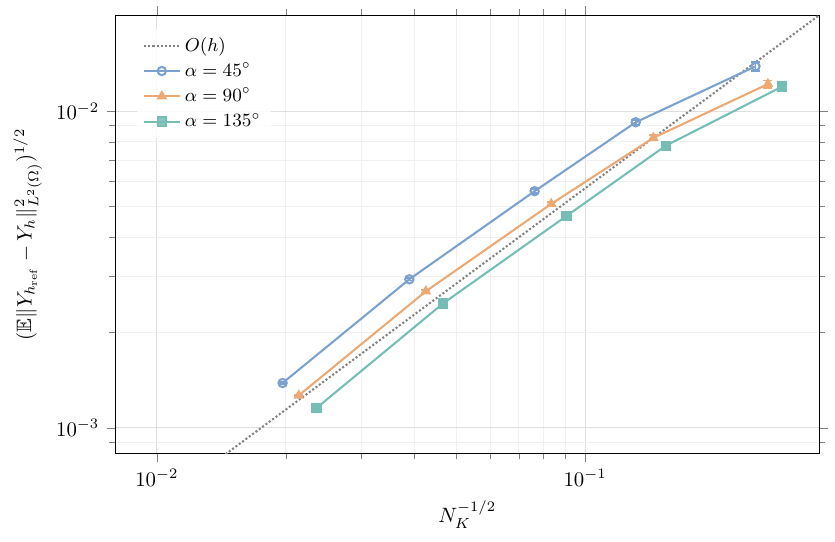}
    \caption{Representative convergence curves}
  \end{subfigure}
  \hfill
  \begin{subfigure}[t]{0.49\textwidth}
    \centering
    \includegraphics[width=\linewidth]
      {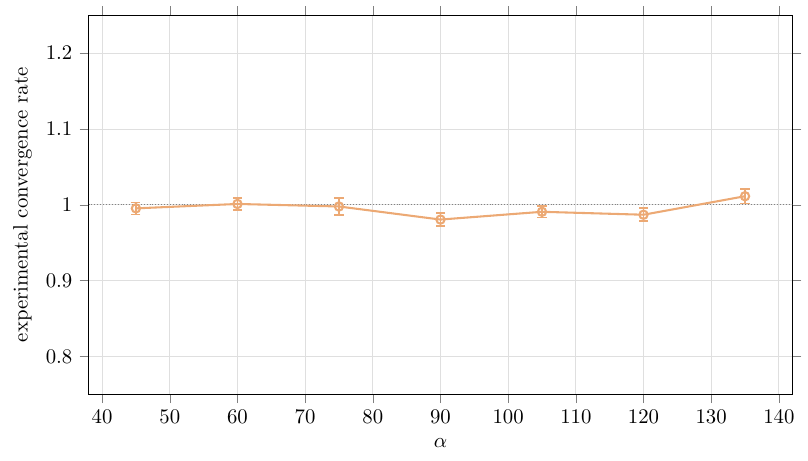}
    \caption{Fitted exponent over the angle sweep}
  \end{subfigure}
  \caption{
    Strong stochastic convergence on the reentrant polygonal domains.
    Panel~(A) shows the coupled fine-reference RMS $L^2$ difference
    against $N_K^{-1/2}$ for representative mouth openings; the
    dotted line shows the first-order rate.  Panel~(B) shows the
    fitted stochastic exponent for all seven mouth openings.  Error
    bars show bootstrap standard errors of the RMS differences in
    panel~(A) and of the fitted exponent in panel~(B).
  }
  \label{fig:pacman_stochastic}
\end{figure}

The experiments recover the two convergence behaviours established
in the analysis.  On regular polygonal and polyhedral domains, the
continuum estimator exhibits the dimension-dependent rate
$h^{2-d/2}$ across several polytopic mesh families.  Across the
reentrant angle sweep, the coupled stochastic approximation exhibits
the first-order rate of
Theorem~\ref{thm:reentrant_spde_convergence}.

\FloatBarrier

\section*{Acknowledgments}

TP is supported by the EPSRC programme grant Mathematics of Radiation
Transport (MaThRad) EP/W026899/2 which is gratefully acknowledged.

\printbibliography

\end{document}